     \documentclass[12pt]{amsart}
\usepackage{}

\usepackage{amsmath}
\usepackage{amsfonts}
\usepackage{amssymb}
\usepackage{graphicx}
\usepackage[all,cmtip]{xy}           
\usepackage{bm}
\usepackage{bbm}

\usepackage{bbding}
\usepackage{txfonts}
\usepackage{amscd}
\usepackage[all]{xy}
\usepackage[shortlabels]{enumitem}
\usepackage{ifpdf}
\ifpdf
\usepackage[colorlinks,final,backref=page,hyperindex]{hyperref}
\else
\usepackage[colorlinks,final,backref=page,hyperindex,hypertex]{hyperref}
\fi
\usepackage{tikz}
\usetikzlibrary{positioning}
\usepackage[active]{srcltx}
\usepackage{caption}
\usepackage{tikz-cd}
\usepackage{dsfont}
\usepackage{mathrsfs}

\makeatletter

\theoremstyle{definition}

\newtheorem{thm}{Theorem}[section]
\newtheorem{lem}[thm]{Lemma}
\newtheorem{cor}[thm]{Corollary}
\newtheorem{pro}[thm]{Proposition}
\newtheorem{ex}[thm]{Example}
\newtheorem{rmk}[thm]{Remark}
\newtheorem{defi}[thm]{Definition}

\newcommand {\emptycomment}[1]{}

\newcommand{\be }{\begin{equation}}
	\newcommand{\ee }{\end{equation}}

\newcommand{\g}{\mathfrak g}

\newcommand{\huaL}{\mathcal{L}}

\newcommand{\huaV}{\mathcal{V}}

\newcommand{\huaP}{\mathcal{P}}

\newcommand{\huaH}{\mathcal{H}}

\newcommand{\huaO}{{\mathcal{O}}}

\newcommand{\frkC}{\mathfrak C}

\newcommand{\frkM}{\mathfrak M}
\newcommand{\frkN}{\mathfrak N}

\newcommand{\br}[1]{   [ \cdot,    \cdot  ]   }
\newcommand{\id}{\mathrm{id}}

\newcommand{\Hom}{\mathrm{Hom}}
\newcommand{\End}{\mathrm{End}}

\newcommand{\K}{\mathbb{K}}

\newcommand{\ot}{\otimes}

\newcommand{\s}{\mathbb{S}}

\newcommand{\op}{\mathcal{P}}
\newcommand{\opup}{\mathcal{P}^{\textup{!}}}
\newcommand{\ceo}{{\rm CE}^{\bullet}}
\newcommand{\Sh}{{\rm Sh}}

\newcommand{\dsN}{\mathds{N}}
\newcommand{\ce}{{\rm CE}}

\newcommand{\ord}{{\rm ord}}

\newcommand{\scrC}{{\mathscr{C}}}

\newcommand{\ff}{f_{(1,1)}}
\newcommand{\nov}{{\rm Nov}}
\newcommand{\ome}{\omega}
\newcommand{\lin}{{\rm Lin}}
\newcommand{\cc}{\circ}
\newcommand{\pl}{{\rm pLie}}

\begin{document}
	
\title[Cohomology theory of Novikov algebras and applications]{Cohomology theory of Novikov algebras and applications}

\author{Pavel Kolesnikov}
\address{Sobolev Institute of Mathematics, Acad. Koptyug ave. 4, Novosibirsk, Russia}
\email{pavelsk77@yandex.ru}

\author{Yue Li}
\address{Department of Mathematics, Jilin University, Changchun 130012, Jilin, China}
\email{liyue25@mails.jlu.edu.cn}

\author{Yunhe Sheng}
\address{Department of Mathematics, Jilin University, Changchun 130012, Jilin, China}
\email{shengyh@jlu.edu.cn}

\author{Nanyan Xu}
\address{Department of Mathematics, Jilin University, Changchun 130012, Jilin, China}
\email{xuny23@jlu.edu.cn}

\begin{abstract}
	In this paper, first we give a new characterization of the cohomology of pre-Lie algebras using the Chevalley-Eilenberg cohomology associated to a morphism from the operad of Lie algebras to Hadamard product of the operad of pre-Lie algebras and its Koszul dual operad. Then we apply the same approach to study the cohomology of Novikov algebras, and give the cochain complex explicitly. The cochain complex of the underlying pre-Lie algebra is shown to be isomorphic to
	 the quotient of the cochain complex of a Novikov algebra. Consequently, there is a long exact sequence connecting the cohomologies of a Novikov algebra and the underlying pre-Lie algebra. The cohomology of a Novikov algebra with coefficients in a representation is introduced using pseudo-tensor categories. As applications, we show that infinitesimal deformations and abelian extensions are classified by the second cohomology groups with different coefficients. Various examples are given to illustrate the difference between the cohomology of a Novikov algebra and that of the underlying pre-Lie algebra.
\end{abstract}

\renewcommand{\thefootnote}{}


\maketitle

\tableofcontents

\allowdisplaybreaks

\section{Introduction}

This paper aims to provide an explicit description of the cochain complex associated with Novikov algebras, establish connections between the cohomology of a Novikov algebra and that of the underlying pre-Lie algebra, and present its applications to deformation and extension problems. This lays the foundation for future study on the general deformation and homotopy theory of Novikov algebras.

\subsection{Novikov Algebras}
Novikov algebras constitute a distinguished class of nonassociative algebras, whose discovery was deeply rooted in the study of Hamiltonian operators in the formal
calculus of variations \cite{GD,GD1}. In a closely related direction, Dubrovin and Novikov developed the Hamiltonian formalism for one-dimensional dynamic systems of hydrodynamic type and introduced a class of local Poisson brackets whose coefficients admit a differential geometric interpretation \cite{Du-No}. Subsequently, Balinsky and Novikov related linear Poisson brackets of hydrodynamic type to finite-dimensional nonassociative algebraic structures satisfying the identities that are now called Novikov algebras \cite{BN}, where the terminology ``Novikov algebra'' was adopted by Osborn \cite{Osborn1992,O1994}.

Zel'manov gave a fundamental structure theory of finite-dimensional Novikov algebras over an algebraically closed field of characteristic 0 \cite{Z}. Osborn developed the theory of modules for Novikov algebras \cite{O1995}, while Xu obtained classification results for simple Novikov algebras and their irreducible modules in both prime characteristic and characteristic zero \cite{XuXP1996,XuXP2001}.
It was shown in \cite{XuXP2000} that Novikov algebras are closely related to a class of Lie conformal algebras arising from the singular part of the operator product expansion of chiral fields in conformal field theory. Bai and Meng have done an in-depth study of the realization of Novikov algebras \cite{BaiM,BaiM2}. Free Novikov algebras have also attracted considerable attention. Their bases were described in terms of trees \cite{DL}, and module structures of a free Novikov algebra over permutation group and general linear group are investigated in \cite{DI}. Another development is the study of Novikov bialgebras, which are related through affinization to infinite-dimensional Lie bialgebras \cite{H-B-G}. Recently, Bruned and Dotsenko related multi-indices arising from singular stochastic partial differential equations to free multi-Novikov algebras \cite{Br-Do}. A unified approach to the structure theory of simple and semisimple Novikov algebras over an arbitrary field has been established in \cite{ZhP2025}. Together, these results illustrate the continuing development of the structural theory of Novikov algebras.

Novikov algebras are special pre-Lie algebras, which are of considerable independent interest in their own right. See  \cite{Bur} for a comprehensive survey of pre-Lie algebras. 
Beyond the properties inherited from the pre-Lie structure, the additional Novikov identity gives rise to a rich theory specific to this class of algebras.

\vspace{-3mm}

\subsection{Cohomology theory of algebraic structures}

Cohomology theory is the primary mathematical tool for studying deformation problems, classifying extensions, and computing invariants of algebraic structures.
The classical cohomology theories of associative algebras, Lie algebras and commutative algebras were developed by Hochschild \cite{Hor}, Chevalley-Eilenberg \cite{CE1948} and  Harrison \cite{Har}. Later, Gerstenhaber \cite{Ger1964} and Nijenhuis-Richardson \cite{NR} developed the deformation theories of associative algebras and Lie algebras respectively and established the relation with corresponding cohomologies.
Loday and  Pirashvili \cite{Loday} established the cohomology theory for Leibniz algebras. Dzhumadil'daev developed the cohomology theory for right-symmetric algebras (right pre-Lie algebras) with coefficients in their modules, related it to the cohomology of the associated Lie algebras, and established the corresponding deformation theory \cite{D1999}.

Operad theory provides a unified framework for describing algebraic structures in terms of their operations and defining relations, and it also gives a general approach to the construction of cohomology theories for many types of algebras.
Ginzburg and Kapranov established Koszul duality for quadratic operads and showed that there is always a morphism from the operad $Lie$ to the Hadamard product of a quadratic operad $\mathcal{P}$ and  its Koszul dual operad $\mathcal{P}^{\textup{!}}$ \cite{GK}. Balavoine further developed the deformation and cohomological theory of algebras governed by quadratic operads, providing a general operadic framework for studying their infinitesimal deformations and related cohomological structures \cite{Bala1995,Bala1998}. More precisely, given a morphism from an operad $Lie$ to an operad $\huaO$, one can construct a Chevalley-Eilenberg cochain complex $\ce^\bullet(\mathcal{O})$. Combined with Ginzburg and Kapranov's result, given a $\huaP$-algebra $V$, equivalently a morphism from the operad $\huaP$ to the operad $\End_V$, one obtain a cochain complex $\ce^\bullet(\mathcal{P}^{\textup{!}}\underset{\rm H}{\ot}\End_V)$. Applying this approach to associative algebras, Lie algebras, commutative algebras and Leibniz algebras, one can recover the aforementioned  classical cohomology theories respectively \cite[Section 3]{Bala1995}.
Furthermore, cohomology with coefficients in arbitrary representations was developed intrinsically in the setting of pseudo-tensor categories  \cite{BDK,BKV,BD}.
This approach provides us the most natural explanation of what is a cohomology for conformal and vertex algebras \cite{BKV,DBKH}.

 \vspace{-3mm}

\subsection{Establishing the cohomology theory for Novikov algebras via operadic theory}
In this paper, we study the cohomology of Novikov algebras using the operadic theory.
Based on the fact that the Koszul dual of the operad $Nov$ of left Novikov algebras is the operad $rNov$ of right Novikov algebras \cite{D}, applying Balavoine and Bakalov-Kac-Voronov's approaches, one can obtain the Chevalley-Eilenberg cochain complex $\ce^\bullet(rNov\underset{\rm H}{\ot}\End_V)$ associated to a Novikov algebra $V$. Note that the operad $rNov$ can be realized as a suboperad of the operad of commutative associative differential algebras, it follows that the operad $rNov$ has a good basis (see \eqref{eq:basis-rNov}), which enables us to extract the cochain complex of Novikov algebras from the above Chevalley-Eilenberg cochain complex. This is the main contribution of the paper. Compared with pre-Lie algebras, Novikov algebras satisfy the additional right-commutativity identity, which leads to a substantially more involved structure of the cochain complex and its differential.

To better understand the relation between the cohomology of a Novikov algebra and  that of  the underlying pre-Lie algebra, we first apply the above approach to pre-Lie algebras, and illustrate how one can obtain the cochain complex of a pre-Lie algebra from the Chevalley-Eilenberg cochain complex $\ce^n(Perm\underset{\rm H}{\ot}\End_V)$. Here $Perm$ is the operad of perm algebras, which is the  Koszul dual operad of the operad $preLie$ of pre-Lie algebras. Note that a perm algebra is naturally a right Novikov algebra, equivalently, there is a surjective morphism from the operad $rNov$ to $Perm$. This morphism naturally induces a cochain map from $\ce^\bullet(rNov\underset{\rm H}{\ot}\End_V)$ to $\ce^\bullet(Perm\underset{\rm H}{\ot}\End_V)$, which shows that the cochain complex of the underlying pre-Lie algebra is isomorphic to the  quotient  of the cochain complex of a Novikov algebra. The short exact sequence of cochain complex naturally induces a long exact sequence of cohomology groups, which connects the cohomologies between a Novikov algebra and the underlying pre-Lie algebra.

We further use the pseudo-tensor category framework to give the cohomology of a Novikov algebra with coefficients in a representation. Applications of the cohomology theory are given to classify infinitesimal deformation and abelian extension problems using the second cohomology groups with different coefficients.  Note that the deformation equations were already given in Bai and Meng's work \cite{BaiM2} in the study of realization of certain Novikov algebras without using the cohomology theory.  On the other hand,
 recently Peng and Tan developed a Chevalley-Eilenberg type cohomology for Novikov algebras \cite{PT2024} from the associated Lie-algebraic viewpoint rather than directly from the defining relations of Novikov algebras. Moreover, the classification of abelian extensions requires an additional condition on second cohomology classes, showing that the full Novikov structure is not encoded by the second cohomology group alone.

 Our research open a gate to study the general deformation and homotopy theory of Novikov algebras. With the explicit deformation complex given in this paper, we conjecture that there is a graded Lie algebra structure on this deformation complex whose Maurer-Cartan elements corresponds to Novikov algebra structures.

\subsection{Organization of the paper and notations.} The paper is organized as follows. In Section \ref{sec:pL}, we give a new characterization of the cohomology of pre-Lie algebras using the Chevalley-Eilenberg cohomology associated to the morphism from the operad $Lie$ to the operad $Perm\underset{\rm H}{\ot} preLie$, the Hadamard product of the operad $Perm$ and the operad $preLie$. In Section \ref{sec:nov}, we apply this operadic approach to give explicit descriptions of the cochain complex for Novikov algebras. In particular, the cochain complex for the underlying pre-Lie algebra is shown to be the quotient complex of the cochain complex for the Novikov algebra. In Section \ref{sec:def}, we study infinitesimal deformations of a Novikov algebra using the established cohomology theory, and show that infinitesimal deformations of a Novikov algebra are classified by the second cohomology group. In Section \ref{sec:rep}, we study the cohomology of a Novikov algebra with coefficients in a representation using the pseudo-tensor category framework. In Section \ref{sec:ext}, we show that abelian extensions of Novikov algebras are classified by the second cohomology group.

Throughout this paper, unless otherwise specified, $\mathds{K}$ is a field of characteristic $0$. All vector spaces and algebras are over $\mathds{K}$. Denote the set of all integers by $\mathds{Z}$.  The identity map is denoted by $\id$, and $\dsN=\{0,1,2,\ldots\}$. For any $n\in \mathds{N}^{\geq1}$, the set $[n]=\{1,2,\cdots,n\}$.

 For any integer partition $i_1+\cdots+i_k=n$ and any $\sigma_1\in \s_{i_1},\cdots, \sigma_k\in \s_{i_k}$, the notation $\sigma_1\times\cdots\times\sigma_k$ is an element of $\s_n$ given by
\small{
\begin{align}\label{def-aaaSn}
\begin{pmatrix}
1          &\cdots&i_1       & i_1+1           &\cdots& i_1+i              &\cdots    &n\\
\sigma_1(1)&\cdots&\sigma_1(i_1)& i_1+\sigma_2(1) &\cdots& i_1+\sigma_2(i_2) &\cdots    &i_1+\cdots+i_{k-1}+\sigma_{k}(i_k)
			\end{pmatrix}
\end{align}}
The set of all permutations of this form is denoted by $\mathbb{S}_{i_1}\times\cdots\times\mathbb{S}_{i_k}$.

A shuffle of type $(i_1, i_2, \cdots, i_k)$ with $i_1 + \cdots + i_k = n$ is a permutation $\sigma \in \mathbb{S}_n$ such that
\begin{align*}
	 \sigma(1)<\cdots<\sigma(i_1), \quad
	  \sigma(i_1+1)<\cdots<\sigma(i_1+i_2), \quad
	  \quad  \cdots,  \quad
	  \sigma(i_1+\cdots+i_{k-1}+1)<\cdots<\sigma(n).
\end{align*}
The set of all such permutations is denoted by $\operatorname{Sh}(i_1,i_2,\cdots,i_k)$. Denote $\Sh^{-1}(i_1,i_2,\cdots,i_k):=\{\sigma\in\s_n\mid\sigma^{-1}\in\Sh(i_1,i_2,\cdots,i_k)\}.$

For distinct elements $i_1,\ldots,i_k\in\{1,\ldots,n\}$, denote by
$(i_1\,i_2\,\cdots\,i_k)\in\s_n$ the cycle permutation satisfying
\[
i_1\mapsto i_2\mapsto\cdots\mapsto i_k\mapsto i_1,
\]
and fixing all elements outside $\{i_1,i_2,\cdots,i_k\}$.

\section{A new characterization of the cohomology of pre-Lie algebras}\label{sec:pL}
In this section, we give a new characterization of the cohomology of pre-Lie algebras using the Chevalley-Eilenberg cohomology associated to the morphism of operads  $Lie \to Perm\underset{\rm H}{\ot} preLie$ .

 \begin{defi}
 	A vector space $\g$ with a binary operator $\diamond:\g\ot\g\to \g$ is called a (left) \textbf{pre-Lie algebra} if for all $x,y,z\in\g$,
 	\begin{align*}
 	(x\diamond y)\diamond z-x\diamond(y\diamond z)=(y\diamond x)\diamond z-y\diamond(x\diamond z).
 	\end{align*}
 \end{defi}

\begin{defi}\cite{D1999,GuLST}.
	Let $(\g,\diamond)$ be a pre-Lie algebra. The pair
	$(\frkC_{\pl}^\bullet(\g),d_\pl)$ is called the cochain complex
	of the pre-Lie algebra $(\g,\diamond)$, where the space of $n$-cochains
	$\frkC_{\pl}^n(\g)$ is defined by
	\begin{align}\label{eq:pLie-n-cochain}
		\frkC_{\pl}^n(\g):=\Hom(\wedge^{n-1}\g\otimes\g,\g),
	\end{align}
and	the coboundary operator
$
d_\pl^n:\frkC_{\pl}^n(\g)\longrightarrow\frkC_{\pl}^{n+1}(\g)
$
is defined as follows: for all $f\in\frkC_{\pl}^n(\g)$ and
$x_1,\ldots,x_{n+1}\in\g$,
\begin{align*}
(d_\pl^n f)(x_1,\ldots,x_n;x_{n+1})={}&\sum_{i=1}^n(-1)^{i+1}x_i\diamond f(x_1,\ldots,\hat{x_i},\ldots,x_n;x_{n+1})\\
&+\sum_{i=1}^n(-1)^{i+1}f(x_1,\ldots,\hat{x_i},\ldots,x_n;x_i)\diamond x_{n+1}\\
&+\sum_{i=1}^n(-1)^if(x_1,\ldots,\hat{x_i},\ldots,x_n;x_i\diamond x_{n+1})\\
&+\sum_{1\leq i<j\leq n}(-1)^{i+j}f(x_i\diamond x_j-x_j\diamond x_i,x_1,\ldots,\hat{x_i},\ldots,\hat{x_j},\ldots,x_n;x_{n+1}).
\end{align*}
\end{defi}

Perm algebras are closely related to pre-Lie algebras. A \textbf{perm algebra} $(A,\cdot)$ is a vector space $A$ with a binary operation $\cdot:A\ot A\to A$ satisfying
\[
(a_1\cdot a_2)\cdot a_3 = (a_2\cdot a_1)\cdot a_3 = a_1\cdot(a_2\cdot a_3),\quad \forall a_1,a_2,a_3\in A.
\]

In the sequel, we use the Chevalley-Eilenberg cohomology associated to certain morphism of operads to characterize the cohomology of pre-Lie algebras. First we recall some basic facts about symmetric operads.  Usually the word symmetric is omitted in this paper.
Recall that an $\s$-module is a family of vector spaces $M=\{M(n)\}_{n\geq 1}$, where each $M(n)$ is a right $\mathds{K}[\s_n]$-module. A morphism of $\s$-modules $f : M\to M'$ is a family of right $\mathds{K}[\s_n]$-module maps $\{f_n:M(n)\to M'(n)\}_{n\geq 1}$.
\begin{defi}
An \textbf{operad} is an $\s$-module $\mathcal{P} = \{\mathcal{P}(n)\}_{n \geq 1}$ with a composition map for all $n\geq 1$ and $i_1,\cdots,i_n\geq 1$
\begin{align*}
\circ : \mathcal{P}(n) \otimes \mathcal{P}(i_1) \otimes \cdots \otimes \mathcal{P}(i_n) \to \mathcal{P}(i_1 + \cdots + i_n)\\
(f;g_1,\cdots,g_n)\mapsto f\circ(g_1,\cdots,g_n)
\end{align*}
and an element $\id \in \mathcal{P}(1)$ which satisfy the following axioms:
for all $f\in\mathcal{P}(n)$, $g_1\in \mathcal{P}(i_1),\cdots,g_n\in \mathcal{P}(i_n)$, $h_{s,t}\in \mathcal{P}(j_{s,t})$, $1\geq s\geq n$, $1\geq t\geq i_s$,
\begin{itemize}
\item {\bf Associativity}:
\begin{align*}
&\quad f\circ\bigl(g_1\circ(h_{1,1},\cdots,h_{1,i_1}),\cdots, g_n\circ(h_{n,1},\cdots,h_{n,i_n})\bigr)\\
&=\bigl(f\circ(g_1,\cdots,g_n)\bigr)\circ(h_{1,1},\cdots,h_{1,i_1},\cdots,h_{n,1},\cdots,h_{n,i_n});
\end{align*}
\item {\bf Identity}:
$
\id\circ f=f=f\circ(\id,\cdots,\id);
$
\item {\bf Equivariance}:
(i)  for $\sigma\in\s_n$,
$
	f^\sigma\circ (g_{\sigma(1)},\cdots,g_{\sigma(n)})=(f\circ(g_1,\cdots,g_n))^{\sigma\langle i_{\sigma(1)},\cdots,i_{\sigma(n)}\rangle},
$
where $\sigma\langle i_{\sigma(1)},\cdots,i_{\sigma(n)}\rangle$ is the block permutation associated to $\sigma$;

(ii) for $\tau_1\in \s_{i_1},\cdots, \tau_n\in \s_{i_n}$,
$
	f\circ(g_1^{\tau_1},\cdots,g_n^{\tau_n})=\bigl(f\circ(g_1,\cdots,g_n)\bigr)^{\tau_1\times\cdots\times\tau_n}.
$
\end{itemize}
\end{defi}

Let $\mathcal{P},\mathcal{P}^{\prime}$ be two operads. A \textbf{morphism of operads} from $\mathcal{P}$ to $\mathcal{P}^{\prime}$ is a morphism of $\s$-modules $\alpha:\mathcal{P}\to \mathcal{P}'$ satisfying $\alpha(\id)=\id'$ and
\begin{align*}
\alpha(f\circ(g_1,\cdots,g_n))=\alpha(f)\circ(\alpha(g_1),\cdots,\alpha(g_n)).
\end{align*}

For any vector space $V$, the Endomorphism operad $\End_V$ is given by
$
\End_V(n)=\Hom(V^{\ot n},V),
$
the right action of $\sigma\in\mathbb{S}_n$ on $f\in\End_V(n)$ is given by
\begin{eqnarray}\label{action-on-map}
	f^\sigma(x_1,\cdots,x_n)=f(x_{\sigma^{-1}(1)},\cdots,x_{\sigma^{-1}(n)}),\quad\forall x_1,\cdots,x_n\in V,
\end{eqnarray}
and the composition map is given by
\begin{align*}
	\circ: \End(n) \otimes \End(i_1) \otimes \cdots \otimes \End(i_n) &\to \End(i_1 + \cdots + i_n)\\
	(f;g_1,\cdots,g_n)&\mapsto f\circ(g_1,\cdots,g_n):=f(g_1\ot \cdots\ot g_n).
\end{align*}

\begin{defi}\label{def-P-algebra}
A {\bf $\mathcal{P}$-algebra} structure on a vector space $V$ is a morphism of operads $\mathcal{P}\to \End_V$.
\end{defi}

Now we recall the cochain complex associated with certain operad morphism.

\begin{defi}\label{def-Lie-operad}
	The operad $Lie$ encoding Lie algebras is generated by a single antisymmetric element $\nu$ of arity $2$ satisfying the Jacobi identity:
	\begin{align*}
		\nu\cc(\nu, \id)+(\nu\cc(\nu,\id))^{(1\,2\,3)}+(\nu\cc(\nu,\id))^{(1\,3\,2)}=0.
	\end{align*}
\end{defi}

\begin{pro}\cite{BKV,Bala1995}\label{CEO-cochain}
	Let $\huaO$ be an operad and $\alpha:Lie\to\huaO$ a morphism of operads.
	The space of Chevalley-Eilenberg $n$-cochains of $\mathcal{O}$ is defined by
	\begin{align}\label{def-ce}
		\ce^n(\mathcal{O})=\{f\in \mathcal{O}(n)\mid f^\sigma=(-1)^\sigma f,\,\forall\sigma\in \mathbb{S}_n\}.
	\end{align}

For $n\geq 1$, define an operator $d_{\ce}^n:\ce^n(\mathcal{O})\to \ce^{n+1}(\mathcal{O})$ by
	\begin{eqnarray}\label{dce}
		&&d^n_{\rm CE}(f)=\underset{\sigma\in\Sh^{-1}(1,n)}{\sum}(-1)^{\sigma}\big(\alpha(\nu)\circ (\id,f)\big)^{\sigma}-\underset{\sigma\in\Sh^{-1}(2,n-1)}{\sum}(-1)^{\sigma}\big(f\circ(\alpha(\nu),\id,\cdots,\id)\big)^{\sigma},
	\end{eqnarray}
where $\Sh(k+1,n-k)$ is the set of all $(k+1,n-k)$-shuffle permutations from $S_{n+1}$.
	Then $(\ceo(\huaO),d^\bullet_{\ce})$ is a cochain complex.
\end{pro}
For example,  for $f\in \ce^1(\mathcal{O})$, we have
$$
d^1_\ce(f)=\alpha(\nu)\circ(\id,f)-\big(\alpha(\nu)\circ (\id,f)\big)^{(1\,2)}-f\circ \alpha(\nu);
$$
for $f\in \ce^2(\mathcal{O})$, we have
\begin{align*}
	d^2_\ce(f)&=\alpha(\nu)\circ(\id,f)-\big(\alpha(\nu)\circ (\id,f)\big)^{(1\,2)}+\big(\alpha(\nu)\circ (\id,f)\big)^{(1\,2\,3)}\\
	&\quad-f\circ(\alpha(\nu),\id)+\big(f\circ(\alpha(\nu),\id)\big)^{(2\,3)}-\big(f\circ(\alpha(\nu),\id)\big)^{(1\,3\,2)}.
\end{align*}

Let $\g$ be a Lie algebra. Equivalently, there is a morphism of operads $Lie\to\End_\g$. The space $\ce^n(\End_\g)$ is precisely $\Hom(\wedge^n\g,\g)$, and the operator $d_{\rm CE}$ coincides with the classical Chevalley-Eilenberg coboundary operator. See \cite{BKV} for further details.

Let $\mathcal{P},\mathcal{Q}$ be two operads. The \textbf{Hadamard product} $\mathcal{P}\underset{\rm H}{\ot}\mathcal{Q}$ of the operads $\mathcal{P}$ and $\mathcal{Q}$ is given by
$$
(\mathcal{P}\underset{\rm H}{\ot}\mathcal{Q})(n):=\mathcal{P}(n)\ot \mathcal{Q}(n).
$$
The right action of $\s_n$ on $(\mathcal{P}\underset{\rm H}{\ot}\mathcal{Q})(n)$ is given by the diagonal action:
\begin{eqnarray*}
	(f\ot g)^\sigma=f^\sigma\ot g^\sigma,\quad f\in \mathcal{P}(n),\,g\in\mathcal{Q}(n),\,\sigma\in\s_n.
\end{eqnarray*}
The composition map is given by
\begin{align*}
	\circ: (\mathcal{P}\underset{\rm H}{\ot}\mathcal{Q})(n) \otimes (\mathcal{P}\underset{\rm H}{\ot}\mathcal{Q})(i_1) \otimes \cdots \otimes (\mathcal{P}\underset{\rm H}{\ot}\mathcal{Q})(i_n) &\to (\mathcal{P}\underset{\rm H}{\ot}\mathcal{Q})(i_1 + \cdots + i_n)\\
	(f\ot g;f_1\ot g_1,\cdots,f_n\ot g_n)&\mapsto (f\circ(f_1,\cdots,f_n))\ot (g\circ(g_1, \cdots, g_n)).
\end{align*}

\begin{thm}\cite{GK,LV}\label{Lie-pp}
	Let $\mathcal{P}$ be a binary quadratic operad and $\mathcal{P}^{\textup{!}}$ its Koszul dual operad. Then there is a morphism of operads $Lie\to \op\underset{\rm H}{\ot} \opup$ as follows:
	\begin{eqnarray*}
		&&\nu\mapsto \sum\limits_{e} e\ot e^*, \quad \nu\in Lie(2),
	\end{eqnarray*}
	where the sum is over a basis $\{e\}$ of $\op(2)$ and $\{e^*\}$ is the dual basis of $\op(2)^*$. Under the canonical identification $\op(2)^*\otimes\mathrm{sgn}_2=\op(2)^\vee=\op^!(2)$, we regard $e_i^*$ as elements of $\op^!(2)$.
\end{thm}

Denote by $preLie$ and $Perm$   the operads of pre-Lie algebras and perm algebras respectively. Recall that the Koszul dual of the $Perm$ operad is the $preLie$ operad \cite{CL}:
\begin{eqnarray}\label{preLie-Perm-dual}
&&{Perm}^! = preLie.
\end{eqnarray}
Moreover, the free perm algebra on a countable set $X=\{x_1, x_2, \dots\}$ has a basis consisting of all monomials of the form \cite{C}
\[
x_1\cdots\widehat{x_i}\cdots x_n\,x_i, \qquad n\ge 1,\ 1\le i\le n.
\]
Equivalently, the space $Perm(n)$ has a basis
\[
\{\, {\xi_n}^\sigma \mid \sigma\in\Sh^{-1}(n-1,1) \,\},
\]
where $\xi_n(x_1,\dots,x_n)=x_1x_2\cdots x_n$.
Note that $\xi_n$ is invariant under the right action of $\s_{n-1}\times \s_1\subset \s_n$ due to the left symmetry of the perm algebra.

By Theorem \ref{Lie-pp} and \eqref{preLie-Perm-dual}, we obtain a morphism of operads $Lie \to Perm\underset{\rm H}{\ot} preLie$ under which the image of $\nu$ is given by
\begin{align*}
\xi_2\ot {\xi_2}^*+{\xi_2}^{(1\,2)}\ot ({\xi_2}^{(1\,2)})^*,
\end{align*}
which is equal to $\xi_2\ot {\xi_2}^*-{\xi_2}^{(1\,2)}\ot ({\xi_2}^*)^{(1\,2)},$
where $\{\xi_2, {\xi_2}^{(1\,2)}\}$ is a basis of $Perm(2)$.

Let $(\g,\omega)$ be a pre-Lie algebra. Equivalently, there is a morphism of operads $\beta:preLie\to \End_\g$ sending ${\xi_2}^*\in preLie(2)$ to $\omega\in \End_\g(2)$. Then there is a morphism of operads  $\alpha: Lie\to Perm\underset{\rm H}{\ot}\End_\g$ as follows:
\begin{eqnarray}\label{Lie-to-pL-Endg}
	&&\alpha:Lie\overset{{\rm Th.}~ \ref{Lie-pp}}{\to} Perm\underset{\rm H}{\ot} Perm^!=Perm\underset{\rm H}{\ot} preLie\overset{\id\ot\beta}{\to}Perm\underset{\rm H}{\ot}\End_\g.
\end{eqnarray}
More precisely, $\alpha$ is given by
\begin{align}\label{alph-nu}
	\alpha(\nu)=\xi_2 \ot \omega-{\xi_2}^{(1\,2)}\ot \omega^{(1\,2)}.
\end{align}
Thus by Proposition \ref{CEO-cochain}, we obtain the following cochain complex.
\begin{pro}\label{CE-pL-Endg}
	Let $(\g,\omega)$ be a pre-Lie algebra. Then $(\ceo(Perm\underset{\rm H}{\ot}\End_\g),d_{\rm CE})$ is a cochain complex, where $d^n_{\rm CE}$ is the differential defined by \eqref{dce} for the above morphism $\alpha$.
\end{pro}


Before we establish an isomorphism of cochain complexes between
$(\ceo(Perm\underset{\rm H}{\ot}\End_\g),d_{\rm CE})$ and
$(\frkC_{\pl}^\bullet(\g),d_\pl)$, we need the following two lemmas, where Lemma \ref{decomp-Sh} is a natural generalization of \cite[Lemma 1.3.3]{LV}.

\begin{lem}\label{decomp-Sh}
	Let $n=i_1+i_2+\cdots+i_k$ be a partition of $n$ and
	$\sigma\in\mathbb{S}_n$.
	\begin{itemize}
		\label{Sh-decom}
		\item[\textrm{(i)}]
		There exist unique permutations
		$\sigma'\in\Sh(i_1,i_2,\ldots,i_k)$ and
		$\sigma_j\in\mathbb{S}_{i_j}$, $1\leq j\leq k$, such that
		\begin{align*}
			\sigma
			=\sigma'(\sigma_1\times\sigma_2\times\cdots\times\sigma_k),
		\end{align*}
		where $\sigma_1\times\cdots\times\sigma_k$ is given by
		\eqref{def-aaaSn}.
		
		\label{Sh-inverse-decom}
		\item[\textrm{(ii)}]
		There exist unique permutations
		$\sigma'\in\Sh^{-1}(i_1,i_2,\ldots,i_k)$ and
		$\sigma_j\in\mathbb{S}_{i_j}$, $1\leq j\leq k$, such that
		\begin{align*}
			\sigma
			=(\sigma_1\times\sigma_2\times\cdots\times\sigma_k)\sigma'.
		\end{align*}
	\end{itemize}
\end{lem}

\begin{lem}\label{pf-in-preL}
For $i=1,2,\cdots,n$, we have the following result:
\begin{align}
\xi_2\circ(\id,{\xi_n}^{(n\,\cdots\,i)})&={\xi_{n+1}}^{(n+1\,\cdots\,i+1)}\label{pf-in-preL-1}\\
{\xi_2}^{(1\,2)}\circ(\id, {\xi_n}^{(n\,\cdots\,i)})&={\xi_{n+1}}^{(n+1\,\cdots\,1)}\label{pf-in-preL-2}\\
{\xi_{n}}^{(n\,\cdots\,i)}\circ({\xi_2},\id,\cdots,\id)&={\xi_{n+1}}^{(n+1\,\cdots\,i+1)}\label{pf-in-preL-3}\\
	{\xi_{n}}^{(n\,\cdots\,i)}\circ({\xi_2}^{(1\,2)},\id,\cdots,\id)&=
\bigg\{
\begin{array}{l}
	{\xi_{n+1}}^{(n+1\,\cdots\,1)},\quad i=1,\\[2mm]
	{\xi_{n+1}}^{(n+1\,\cdots\,i+1)},\quad i\neq1.
\end{array}\label{pf-in-preL-4}
\end{align}
\end{lem}

By \eqref{def-ce}, we have
\[
\ce^n(Perm\underset{{\rm H}}{\ot}\End_\g):=\{F\in (Perm\underset{{\rm H}}{\ot}\End_\g)(n)\mid F^\sigma=(-1)^\sigma F,\ \forall \sigma\in\s_n\}.
\]
Since $\{{\xi_n}^\sigma \mid \sigma\in\Sh^{-1}(n-1,1)\}$ is a basis of $Perm(n)$, any element $F\in (Perm\underset{{\rm H}}{\ot}\End_\g)(n)$ can be written as
\begin{align*}
	F=\sum_{\sigma\in\Sh^{-1}(n-1,1)} {\xi_n}^\sigma \ot f_\sigma\in Perm(n)\ot \End_\g(n).
\end{align*}
For any element $F=\sum\limits_{\sigma\in\Sh^{-1}(n-1,1)}{\xi_n}^\sigma\otimes f_\sigma \in \ce^n(Perm\underset{{\rm H}}{\ot}\End_\g)$ ($n\geq 1$), we claim that $f_\id\in \frkC_{\pl}^n(\g)$. Indeed, take $\tau\in\s_{n-1}\times\s_1$. By $F^\tau=(-1)^\tau F$, we obtain
\[
\sum_{\sigma\in\Sh^{-1}(n-1,1)} {\xi_n}^{\sigma\tau}\ot f_\sigma{}^\tau = (-1)^\tau\sum_{\sigma\in\Sh^{-1}(n-1,1)} {\xi_n}^\sigma\ot f_\sigma .
\]
Since $\xi_{n}$ is invariant under the right action of $\s_{n-1}\times\s_1$, Lemma \ref{decomp-Sh} {\rm (ii)} implies that  the set $\{\xi^{\sigma\tau}_n\mid \sigma\in Sh^{-1}(n-1,1)\}$ is precisely the basis $\{{\xi_n}^{\sigma}\mid\sigma\in Sh^{-1}(n-1,1)\}$ of $Perm(n)$.   Therefore, comparing the coefficients of $\xi_n$ in the above equation gives $f_\id{}^\tau=(-1)^\tau f_\id$, which shows that $f_\id$ is alternating in the first $n-1$ variables and hence belongs to $\frkC_{\pl}^n(\g)$. Then the linear map $\Omega_n:\ce^n(Perm\underset{{\rm H}}{\ot}\End_\g)\to \frkC_{\pl}^n(\g)$ defined as follows is well-defined:
\begin{align}
\sum_{\sigma\in\Sh^{-1}(n-1,1)} {\xi_n}^\sigma\ot f_\sigma \longmapsto f_\id.
\end{align}

\begin{pro}\label{Omega}
With the above notations, $\Omega_n:\ce^n(Perm\underset{{\rm H}}{\ot}\End_\g)\to \frkC_{\pl}^n(\g)$ is an isomorphism of vector spaces.
\end{pro}
\begin{proof}
For any $f\in\frkC_{\pl}^n(\g)$, set
\[
F_f := \sum_{\sigma\in\Sh^{-1}(n-1,1)} (-1)^\sigma {\xi_n}^\sigma \ot f^\sigma.
\]
We first check that $F_f\in\ce^n(Perm\underset{\rm H}{\ot}\End_\g)$, i.e. $F_f{}^\tau = (-1)^\tau F_f$ for all $\tau\in\s_n$.
Fix a $\tau\in\s_n$, for any $\sigma\in\Sh^{-1}(n-1,1)$, it is obvious that $\sigma\tau\in\s_n$.
By Lemma~\ref{Sh-decom} (ii), there exist unique $\sigma_{n-1}\in\s_{n-1},\sigma_1\in\s_1$ and $\sigma'\in\Sh^{-1}(n-1,1)$ such that $\sigma\tau = (\sigma_{n-1}\times\sigma_1)\sigma'$.
Using $\xi_n^{\sigma_{n-1}\times\sigma_1}=\xi_n$ and  $f^{\sigma_{n-1}\times\sigma_1}=(-1)^{\sigma_{n-1}\times\sigma_1}f$, we have
\begin{align*}
F_f{}^\tau
&=\sum_{\sigma\in\Sh^{-1}(n-1,1)} (-1)^\sigma {\xi_n}^{(\sigma_{n-1}\times\sigma_1)\sigma'} \ot f^{(\sigma_{n-1}\times\sigma_1)\sigma'} =\sum_{\sigma\in\Sh^{-1}(n-1,1)} (-1)^\sigma (-1)^{\sigma_{n-1}\times\sigma_1} {\xi_n}^{\sigma'} \ot f^{\sigma'} .
\end{align*}
From $\sigma\tau = (\sigma_{n-1}\times\sigma_1)\sigma'$ we have $(-1)^\sigma (-1)^\tau = (-1)^{\sigma_{n-1}\times\sigma_1} (-1)^{\sigma'}$, so $(-1)^\sigma (-1)^{\sigma_{n-1}\times\sigma_1} = (-1)^\tau (-1)^{\sigma'}$.
The uniqueness of $\sigma'$ implies that the map $\sigma\mapsto\sigma'$ is a bijection on $\Sh^{-1}(n-1,1)$, then re-indexing the sum gives
\[
F_f{}^\tau =\sum_{\sigma'\in\Sh^{-1}(n-1,1)} (-1)^\tau (-1)^{\sigma'} {\xi_n}^{\sigma'} \ot f^{\sigma'} = (-1)^\tau F_f .
\]
which implies that the map $\frkC^n_{\pl}(\g)\to\ce^n(Perm\otimes_{\rm H}\End_\g)$ sending $f$ to $F_f$ is well-defined,
and it is obviously the inverse of $\Omega_{n}$ since both compositions yield the identity on the respective spaces. Consequently, $\ce^n(Perm\underset{{\rm H}}{\ot}\End_\g)$ and $\frkC_{\pl}^n(\g)$ are isomorphic as vector spaces.
\end{proof}

Now we are ready to give a new characterization of the cohomology of pre-Lie algebras.

\begin{thm}\label{pre-Lie-cochain}
$\Omega=\{\Omega_n\}_{n\geq 1}:(\ce^\bullet(Perm\underset{{\rm H}}{\ot}\End_\g),d_\ce)\to(\frkC_{\pl}^\bullet(\g),d_\pl)$ is an isomorphism of cochain complexes.
Equivalently, the following diagram
commutes:
\[
\xymatrix@C=4em{
	\frkC_{\pl}^n(\g)
	\ar[r]^-{\Omega_n^{-1}}
	\ar[d]_{d_\pl^n}
	&
	\ce^n(Perm\underset{{\rm H}}{\ot}\End_\g)
	\ar[d]^{d_\ce^n}
	\\
	\frkC_{\pl}^{n+1}(\g)
	&
	\ce^{n+1}(Perm\underset{{\rm H}}{\ot}\End_\g)
	\ar[l]_-{\Omega_{n+1}}
}.
\]
\end{thm}

\begin{proof}
 For $n\geq 1$ and every $f\in\frkC_{\pl}^n(\g)=\Hom(\wedge^{n-1}\g\otimes\g,\g)$, we need to show that
 $$
 d_{\pl}^n
 =
 \Omega_{n+1}\circ d_{\rm CE}^n\circ{\Omega_n}^{-1}.
 $$
 By Proposition \ref{Omega}, we have
	\begin{align*}
		\Omega^{-1}_n(f):=F_f=\sum_{\sigma\in\Sh^{-1}(n-1,1)} (-1)^\sigma {\xi_n}^\sigma \ot f^\sigma.
	\end{align*}
	Since $\Sh^{-1}(n-1,1)=\{(n\,n-1\,\cdots\,i)\mid i=1,\dots,n\}$ and $(-1)^{(n\,n-1\,\cdots\,i)}=(-1)^{n-i}$, we can write
	\begin{align*}
		F_f=\sum_{i=1}^n (-1)^{n-i} {\xi_n}^{(n\,n-1\,\cdots\,i)} \ot f^{(n\,n-1\,\cdots\,i)}.
	\end{align*}
With the abbreviation $\id\otimes\id$ for $\id_{Perm}\otimes\id_{\End_\g}$, we obtain the three parts of $d_{\rm CE}^n(F_f)$:
	\begin{eqnarray*}
d_{\rm CE}^n(F_f)&\overset{\eqref{dce}}=&\sum_{\sigma\in\Sh^{-1}(1,n)}(-1)^{\sigma}\bigl(\alpha(\nu)\circ(\id\ot\id,F_f)\bigr)^{\sigma}
		-\sum_{\sigma\in\Sh^{-1}(2,n-1)}(-1)^{\sigma}\bigl(F_f\circ(\alpha(\nu),\id\ot\id,\cdots,\id\ot\id)\bigr)^{\sigma}\\
		&\overset{\eqref{alph-nu}}=&\underbrace{\sum_{\sigma\in\Sh^{-1}(1,n)}(-1)^{\sigma}\bigl((\xi_n\otimes\omega)\circ(\id\otimes\id,F_f)\bigr)^{\sigma}}_{=:A}\\
		&&\underbrace{-\sum_{\sigma\in\Sh^{-1}(1,n)}(-1)^{\sigma}\bigl(({\xi_n}^{(1\,2)}\otimes\omega^{(1\,2)})\circ(\id\otimes\id,F_f)\bigr)^{\sigma}}_{=:B}\\
		&&\underbrace{-\sum_{\sigma\in\Sh^{-1}(2,n-1)}(-1)^{\sigma}\bigl(F_f\circ(\alpha(\nu),\id\otimes\id,\dots,\id\otimes\id)\bigr)^{\sigma}}_{=:C}.
	\end{eqnarray*}
By the definition of $\Omega_{n+1}$, we now examine $A,B,C$ separately and collect all terms whose $Perm$-component equals $\xi_{n+1}$.
	
\textbf{The term $A$.}
Substituting the expression for $F_f$ into $A$, we obtain
$$
A=\sum\limits_{\sigma\in\Sh^{-1}(1,n)} \sum\limits_{i=1}^{n}(-1)^{\sigma}(-1)^{n-i}\bigl((\xi_n\otimes\omega)\circ(\id\otimes\id, {\xi_n}^{(n\,\cdots\,i)}\ot f^{(n\,\cdots\,i)})\bigr)^{\sigma}.
$$ Note that $\Sh^{-1}(1,n)=\{(1\,2\,\cdots\,j)\mid j=1,\dots,n+1\}$ and $(-1)^{(1\,2\,\cdots\,j)}=(-1)^{j+1}$.
 Then  we have
	\begin{eqnarray*}
		A&=&\sum\limits_{j=1}^{n+1}\sum\limits_{i=1}^{n} (-1)^{j+1}(-1)^{n-i} \bigl((\xi_n\ot\omega)\circ(\id\ot\id,{\xi_n}^{(n\,\cdots\,i)}\ot f^{(n\,\cdots\,i)})\bigr)^{(1\,\cdots\,j)}\\
		&=&\sum\limits_{j=1}^{n+1}\sum\limits_{i=1}^{n} (-1)^{n+j+1-i}\bigl(\xi_2\circ(\id,{\xi_n}^{(n\,\cdots\,i)})\bigr)^{(1\,\cdots\,j)}\ot \bigl(\omega\circ(\id, f^{(n\,\cdots\,i)})\bigr)^{(1\,\cdots\,j)}\\
		&\overset{\eqref{pf-in-preL-1}}{=}&\sum\limits_{j=1}^{n+1}\sum\limits_{i=1}^{n} (-1)^{n+j+1-i}\bigl({\xi_{n+1}}^{(n+1\,\cdots\,i+1)}\bigr)^{(1\,\cdots\,j)}\ot \bigl(\omega\circ(\id, f^{(n\,\cdots\,i)})\bigr)^{(1\,\cdots\,j)}
	\end{eqnarray*}
Therefore, we need to find all $i,j$ such that ${\xi_{n+1}}^{(n+1\,\cdots\,i+1)(1\,\cdots\,j)}=\xi_{n+1}$. Since $\xi_{n+1}$ is invariant under the right action of $\s_n\times\s_1$, Lemma \ref{decomp-Sh} {\rm (ii)} implies that this condition is equivalent to
$
(n+1\,\cdots\,i+1)(1\,\cdots\,j)\in\s_n\times\s_1.
$
In other words, $n+1$ must be fixed by the permutation
$
(n+1\,\cdots\,i+1)(1\,\cdots\,j).
$
Thus, ${\xi_{n+1}}^{(n+1\,\cdots\,i+1)(1\,\cdots\,j)}=\xi_{n+1}$ holds precisely when $i=n,\; j=1,\ldots,n.$
 Therefore, collecting the terms in $A$ with $Perm$-component $\xi_{n+1}$, we obtain
\begin{align}\label{eq-preL-cochain-1}
	\sum_{j=1}^{n}
	(-1)^{j+1}
	\xi_{n+1}\otimes
	\bigl(\omega\circ(\id,f)\bigr)^{(1\,\cdots\,j)}.
\end{align}

\textbf{The term $B$.}
Substituting the expression for $F_f$ into $B$, we obtain
$$
B=-\sum\limits_{\sigma\in\Sh^{-1}(1,n)}\sum\limits_{i=1}^{n}(-1)^{\sigma}(-1)^{n-i}\bigl(({\xi_n}^{(1\,2)}\otimes\omega^{(1\,2)})\circ(\id\otimes\id,{\xi_n}^{(n\,\cdots\,i)}\ot f^{(n\,\cdots\,i)})\bigr)^{\sigma}.
$$
Therefore by Lemma \ref{pf-in-preL} and $\Sh^{-1}(1,n)=\{(1\,2\,\cdots\,j)\mid j=1,\dots,n+1\}$, we have
	\begin{eqnarray*}
		B&=&-\sum\limits_{j=1}^{n+1}\sum\limits_{i=1}^{n} (-1)^{j+1}(-1)^{n-i} \bigl(({\xi_n}^{(1\,2)}\ot \omega^{(1\,2)})\circ(\id\ot\id,{\xi_n}^{(n\,\cdots\,i)}\ot f^{(n\,\cdots\,i)})\bigr)^{(1\,\cdots\,j)}\\
		&=&\sum\limits_{j=1}^{n+1}\sum\limits_{i=1}^{n} (-1)^{n+j-i} \bigl({\xi_2}^{(1\,2)}\circ(\id, {\xi_n}^{(n\,\cdots\,i)})\bigr)^{(1\,\cdots\,j)}\ot \bigl(\omega^{(1\,2)}\circ (\id,f^{(n\,\cdots\,i)})\bigr)^{(1\,\cdots\,j)}\\
		&\overset{\eqref{pf-in-preL-2}}{=}&\sum\limits_{j=1}^{n+1}\sum\limits_{i=1}^{n} (-1)^{n+j-i} {\xi_{n+1}}^{(n+1\,\cdots\,1)(1\,\cdots\,j)}\ot \bigl(\omega^{(1\,2)}\circ (\id,f^{(n\,\cdots\,i)})\bigr)^{(1\,\cdots\,j)}.
	\end{eqnarray*}
Obviously,
${\xi_{n+1}}^{(n+1\,\cdots\,1)(1\,\cdots\,j)}=\xi_{n+1}$ holds precisely when $j=n+1$. Hence, collecting the terms in $B$ with $Perm$-component $\xi_{n+1}$, we obtain
	\begin{eqnarray}\label{eq-preL-cochain-2}
		\sum\limits_{i=1}^n (-1)^{1-i} \xi_{n+1}\ot \bigl(\omega^{(1\,2)}\circ (\id,f^{(n\,\cdots\,i)})\bigr)^{(1\,\cdots\,n+1)}.
	\end{eqnarray}

\textbf{The term $C$.}
Substituting the expression for $F_f$ into $C$, we obtain 	
$$
C=-\sum\limits_{\sigma\in\Sh^{-1}(2,n-1)}\sum\limits_{i=1}^n(-1)^{\sigma}(-1)^{n-i}\bigl(({\xi_n}^{(n\,\cdots\,i)}\ot f^{(n\,\cdots\,i)})\circ(\alpha(\nu),\id\otimes\id,\dots,\id\otimes\id)\bigr)^{\sigma}.
$$
	Note that \begin{align*}
		\Sh^{-1}(2,n-1)&=\big\{\sigma_{s,t}\in\s_{n+1}\mid 1\leq s<t\leq n+1,\sigma(s)=1,\sigma(t)=2,\\
		&\;\;\quad\sigma_{s,t}(1)<\cdots<\sigma_{s,t}(s-1)<\sigma_{s,t}(s+1)<\cdots\sigma_{s,t}(t-1)<\sigma_{s,t}(t+1)<\cdots\sigma_{s,t}(n+1)\big\}.
	\end{align*}
	and $(-1)^{\sigma_{s,t}}=(-1)^{s+t-1}$. Then we have
	\begin{align*}
		C&=\sum\limits_{1\leq s<t\leq n+1}\sum\limits_{i=1}^n(-1)^{s+t+n-i}\bigl(({\xi_n}^{(n\,\cdots\,i)}\ot f^{(n\,\cdots\,i)})\circ(\xi_2 \ot \omega-{\xi_2}^{(1\,2)}\ot \omega^{(1\,2)},\id\otimes\id,\dots,\id\otimes\id)\bigr)^{\sigma_{s,t}}\\
		&=\sum\limits_{1\leq s<t\leq n+1}\sum\limits_{i=1}^n(-1)^{s+t+n-i}\bigl( {\xi_n}^{(n\,\cdots\,i)}\circ(\xi_2,\id,\cdots,\id)\bigr)^{\sigma_{s,t}}\ot \bigl(f^{(n\,\cdots\,i)}\circ(\omega,\id,\cdots,\id)\bigr)^{\sigma_{s,t}}\\
		&\quad\quad-\sum\limits_{1\leq s<t\leq n+1}\sum\limits_{i=1}^n(-1)^{s+t+n-i}\bigl( {\xi_n}^{(n\,\cdots\,i)}\circ({\xi_2}^{(1\,2)},\id,\cdots,\id)\bigr)^{\sigma_{s,t}}\ot \bigl(f^{(n\,\cdots\,i)}\circ(\omega^{(1\,2)},\id,\cdots,\id)\bigr)^{\sigma_{s,t}}.
	\end{align*}
	We now separate the terms with $i=1$ from those with $i=2,\ldots,n$. By \eqref{pf-in-preL-3} and \eqref{pf-in-preL-4}, the terms corresponding to $i=1$ are
\begin{align*}
	&\sum\limits_{1\leq s<t\leq n+1}(-1)^{s+t+n-1}\bigl( {\xi_n}^{(n\,\cdots\,1)}\circ(\xi_2,\id,\cdots,\id)\bigr)^{\sigma_{s,t}}\ot \bigl(f^{(n\,\cdots\,1)}\circ(\omega,\id,\cdots,\id)\bigr)^{\sigma_{s,t}}\\
	&\quad-\sum\limits_{1\leq s<t\leq n+1}(-1)^{s+t+n-1}\bigl( {\xi_n}^{(n\,\cdots\,1)}\circ({\xi_2}^{(1\,2)},\id,\cdots,\id)\bigr)^{\sigma_{s,t}}\ot \bigl(f^{(n\,\cdots\,1)}\circ(\omega^{(1\,2)},\id,\cdots,\id)\bigr)^{\sigma_{s,t}}\\
	&=\sum_{1\leq s<t\leq n+1}
	(-1)^{s+t+n-1}
	{\xi_{n+1}}^{(n+1\,\cdots\,2)\sigma_{s,t}}
	\otimes
	\bigl(f^{(n\,\cdots\,1)}
	\circ(\omega,\id,\dots,\id)\bigr)^{\sigma_{s,t}}\\
	&\quad\quad-
	\sum_{1\leq s<t\leq n+1}
	(-1)^{s+t+n-1}
	{\xi_{n+1}}^{(n+1\,\cdots\,1)\sigma_{s,t}}
	\otimes
	\bigl(f^{(n\,\cdots\,1)}
	\circ(\omega^{(1,2)},\id,\dots,\id)\bigr)^{\sigma_{s,t}},
\end{align*}
while the terms corresponding to $i=2,\ldots,n$ are
\begin{align*}
&\sum\limits_{1\leq s<t\leq n+1}\sum\limits_{i=2}^n(-1)^{s+t+n-i}\bigl( {\xi_n}^{(n\,\cdots\,i)}\circ(\xi_2,\id,\cdots,\id)\bigr)^{\sigma_{s,t}}\ot \bigl(f^{(n\,\cdots\,i)}\circ(\omega,\id,\cdots,\id)\bigr)^{\sigma_{s,t}}\\
&\quad\quad-\sum\limits_{1\leq s<t\leq n+1}\sum\limits_{i=2}^n(-1)^{s+t+n-i}\bigl( {\xi_n}^{(n\,\cdots\,i)}\circ({\xi_2}^{(1\,2)},\id,\cdots,\id)\bigr)^{\sigma_{s,t}}\ot \bigl(f^{(n\,\cdots\,i)}\circ(\omega^{(1\,2)},\id,\cdots,\id)\bigr)^{\sigma_{s,t}}\\
&=\sum_{1\leq s<t\leq n+1}\sum_{i=2}^n(-1)^{s+t+n-i}{\xi_{n+1}}^{(n+1\,\cdots\,i+1)\sigma_{s,t}}
\otimes\bigl(f^{(n\,\cdots\,i)}\circ(\omega,\id,\dots,\id)\bigr)^{\sigma_{s,t}}\\
&\quad\quad-\sum_{1\leq s<t\leq n+1}\sum_{i=2}^n(-1)^{s+t+n-i}{\xi_{n+1}}^{(n+1\,\cdots\,i+1)\sigma_{s,t}}\otimes
	\bigl(f^{(n\,\cdots\,i)}
	\circ(\omega^{(1,2)},\id,\dots,\id)\bigr)^{\sigma_{s,t}}\\
&=\sum_{1\leq s<t\leq n+1}\sum_{i=2}^n
	(-1)^{s+t+n-i}
	{\xi_{n+1}}^{(n+1\,\cdots\,i+1)\sigma_{s,t}}
	\otimes
	\bigl(f^{(n\,\cdots\,i)}
	\circ(\omega-\omega^{(1,2)},\id,\dots,\id)\bigr)^{\sigma_{s,t}}.
\end{align*}
That is,
\begin{align*}
C= 	&\sum_{1\leq s<t\leq n+1}
(-1)^{s+t+n-1}
{\xi_{n+1}}^{(n+1\,\cdots\,2)\sigma_{s,t}}
\otimes
\bigl(f^{(n\,\cdots\,1)}
\circ(\omega,\id,\dots,\id)\bigr)^{\sigma_{s,t}}\\
&\quad-
\sum_{1\leq s<t\leq n+1}
(-1)^{s+t+n-1}
{\xi_{n+1}}^{(n+1\,\cdots\,1)\sigma_{s,t}}
\otimes
\bigl(f^{(n\,\cdots\,1)}
\circ(\omega^{(1,2)},\id,\dots,\id)\bigr)^{\sigma_{s,t}}\\
&+\sum_{1\leq s<t\leq n+1}\sum_{i=2}^n
(-1)^{s+t+n-i}
{\xi_{n+1}}^{(n+1\,\cdots\,i+1)\sigma_{s,t}}
\otimes
\bigl(f^{(n\,\cdots\,i)}
\circ(\omega-\omega^{(1,2)},\id,\dots,\id)\bigr)^{\sigma_{s,t}}.
\end{align*}

  	Clearly, ${\xi_{n+1}}^{(n+1\,\cdots\,2)\sigma_{s,t}}=\xi_{n+1}$ holds precisely when  $t=n+1,1\leqslant s\leqslant n$; ${\xi_{n+1}}^{(n+1\,\cdots\,i+1)\sigma_{s,t}}=\xi_{n+1}$ holds precisely when $i=n,1\leqslant s<t\leqslant n$; and there does not exist any $s,t$ such that ${\xi_{n+1}}^{(n+1\,\cdots\,1)\sigma_{s,t}}=\xi_{n+1}$. Therefore, collecting the terms in $C$ with $Perm$-component $\xi_{n+1}$, we obtain
	\begin{eqnarray}\label{eq-preL-cochain-3}
		&&\sum\limits_{s=1}^n (-1)^{s} \xi_{n+1}\ot \bigl(f^{(n\,\cdots\,1)}\circ(\omega,\id,\cdots,\id)\bigr)^{\sigma_{s,n+1}}\\
		&&\quad+\sum\limits_{1\leqslant s<t\leqslant n}(-1)^{s+t} \xi_{n+1}\ot \bigl(f\circ(\omega-\omega^{(1\,2)},\id,\cdots,\id)\bigr)^{\sigma_{s,t}}.\nonumber
	\end{eqnarray}
	
	In summary, combining \eqref{eq-preL-cochain-1}-\eqref{eq-preL-cochain-3} and Proposition \ref{Omega}, we have
	\begin{align*}
	&\quad\Omega_{n+1}\circ d^n_\ce\circ{\Omega_{n}}^{-1}(f)=\Omega_{n+1}\circ d^n_\ce(F_f)\\
		&=\sum\limits_{j=1}^{n} (-1)^{j+1}\bigl(\omega\circ(\id, f)\bigr)^{(1\,\cdots\,j)}+\sum\limits_{i=1}^n (-1)^{1-i}  \bigl(\omega^{(1\,2)}\circ (\id,f^{(n\,\cdots\,i)})\bigr)^{(1\,\cdots\,n+1)}\\
		&\quad+\sum\limits_{s=1}^n (-1)^{s} \bigl(f^{(n\,\cdots\,1)}\circ(\omega,\id,\cdots,\id)\bigr)^{\sigma_{s,n+1}}+\sum\limits_{1\leqslant s<t\leqslant n}(-1)^{s+t}  \bigl(f\circ(\omega-\omega^{(1\,2)},\id,\cdots,\id)\bigr)^{\sigma_{s,t}},
	\end{align*}
Comparing the resulting expression with the definition of \(d_{\pl}^n\), we conclude that
$$
d_{\pl}^n
=
\Omega_{n+1}\circ d_{\rm CE}^n\circ{\Omega_n}^{-1}.
$$
This completes the proof.
\end{proof}

\begin{rmk}
One may interpret  Theorem~\ref{pre-Lie-cochain} as follows:
the cochain complex construction $\frkC_{\pl}^\bullet $
is a splitting (bi-successor in terms of \cite{BBGN2013})
of ${\rm CE}^\bullet_{\rm Lie} $.
\end{rmk}

\section{Cohomology theory of Novikov algebras}\label{sec:nov}
In this section, we describe explicitly the cochain complex of a Novikov algebra via the Chevalley-Eilenberg cohomology associated to the operad morphism $Lie \to rNov\underset{\rm H}{\ot} Nov$.
In particular, we show that the cochain complex of the underlying pre-Lie algebra is the quotient of the cochain complex of a Novikov algebra. Consequently, there is a long exact sequence of cohomology groups.

\begin{defi}\label{def-l-r-Nov}
A vector space $\g$ with a binary operation $\diamond:\g\ot\g\to\g$ is called a {\bf left Novikov algebra} if for any $x,y,z\in \g$,
\begin{eqnarray}
(x\diamond y)\diamond z-x\diamond (y\diamond z)&=&(y\diamond x)\diamond z-y\diamond (x\diamond z),\label{l-Nov-1}\\
(x\diamond y)\diamond z&=&(x\diamond z)\diamond y.\label{l-Nov-2}
\end{eqnarray}
A vector space $\g$ with a binary operation $\circ:\g\ot\g\to\g$ is called a {\bf right Novikov algebra} if for any $x,y,z\in \g$,
\begin{eqnarray*}
(x\circ y)\circ z-x\circ (y\circ z)&=&(x\circ z)\circ y-x\circ (z\circ y),\\
x\circ (y\circ z)&=&y\circ (x\circ z).
\end{eqnarray*}
\end{defi}
In the sequel, by Novikov algebras we mean left Novikov algebras. A \textbf{homomorphism of Novikov algebras} from $(\g,\diamond)$ to $(\g',\diamond')$ is a linear map $f:\g\to \g'$ satisfying
\begin{align*}
	f(x\diamond y)=f(x)\diamond'f(y),\quad \forall x,y\in\g.
\end{align*}

\begin{rmk}
	Let $(\g, \diamond)$ be a Novikov algebra. Then in particular $(\g,\diamond)$ is a pre-Lie algebra, which is called the underlying pre-Lie algebra.
\end{rmk}

\begin{ex}\cite{GD}\label{ex-rN-N}
	Let $(A,\cdot)$ be a commutative associative algebra with a derivation $D$.
	Fix a scalar $\lambda \in \mathds K$ and define two binary operations $\diamond,\circ:A\otimes A\to A$ as follows:
	\begin{eqnarray}
		a\diamond b&=&a\cdot D(b) + \lambda a\cdot b,\\
		a\circ b&=&D(a)\cdot b + \lambda a\cdot b,\quad  \forall a,b\in A.\label{rNov-Dcom}
	\end{eqnarray}
Then $(A,\diamond)$ is a Novikov algebra and $(A,\circ)$ is a right Novikov algebra.
\end{ex}

\begin{ex}\label{ex-kt-p}
\begin{itemize}
\item[{\rm (i)}] Let $\mathds{K}[t]$ be the polynomial algebra with the ordinary derivation $\frac{d}{dt}$. Then by Example \ref{ex-rN-N} with $\lambda =0$, $(\mathds{K}[t],\diamond)$ is a Novikov algebra, where $\diamond$ is given by
\begin{align*}
t^i\diamond t^j:=t^i\frac{d}{dt} (t^j)=jt^{i+j-1},\quad\forall i,j\in\mathds{N}.
\end{align*}
\item[{\rm (ii)}]  Let $p$ be a prime and $\mathds{K}$ be a field of characteristic $p$. Then by Example \ref{ex-rN-N} with $\lambda =0$, $(\mathds{K}[t]/(t^p),\diamond)$ is a Novikov algebra, where $\diamond$ is given by
\begin{align*}
	t^i\diamond t^j:=t^i\frac{d}{dt} (t^j)=jt^{i+j-1},\quad\forall i,j\in\{0,1,\cdots,p-1\}.
\end{align*}
\end{itemize}
\end{ex}

Denote by $Nov$ and $rNov$ the operads of left and right Novikov algebras respectively. It was shown in \cite{D} that the Koszul dual operad of $rNov$ is $Nov$, that is, $rNov^! = Nov$.

The operad encoding the commutative associative differential algebras is denoted by $ComDer$. The operad $ComDer$ is generated by an operation $D$ of arity $1$ and a symmetric operation $\mu$ of arity $2$ satisfying the relations:
\begin{eqnarray}
	\mu\circ(\mu,\id)&=& \mu\circ(\id,\mu),\\
	D\circ \mu&=&\mu\circ(D,\id)+\mu\circ(\id,D).
\end{eqnarray}
Let $A$ be a $ComDer$-algebra. For any \(a, a_1,a_2\in A\) and \(i\geq 0\), we denote by
$$
a_1a_2:=\mu(a_1,a_2), \qquad a^{(i)}:=D^i(a).
$$
We omit unnecessary brackets due to the associativity of $\mu$.  By \cite{DIU,DL}, the operad $rNov$ admits an explicit realization as the suboperad of $ComDer$. More precisely, for each $n\geq 1$, $rNov(n)$ has a basis as follows:
\begin{align}\label{eq:basis-rNov}
	\big\{\mu_n\circ (D^{p_1}, D^{p_2},\cdots,D^{p_n})\mid (p_1,\cdots,p_n)\in P_n\big\},
\end{align}
where $\mu_n(a_1,a_2,\cdots,a_n)=a_1a_2\cdots a_n$ and the set $P_n$ is defined by
\begin{eqnarray}\label{Pn}
P_n&=&\{(p_1,\cdots,p_n)\in\dsN^n\mid\sum\limits_{i=1}^n p_i=n-1\}.
\end{eqnarray}
The right $\mathbb{S}_n$-module action on $rNov(n)$ is inherited from $ComDer(n)$ as follows:
\begin{align}\label{action-on-rNov}
	\big(\mu_n\circ (D^{p_1}, D^{p_2},\cdots,D^{p_n})\big)^\sigma &={\mu_n}^\sigma\circ (D^{p_{\sigma(1)}}, D^{p_{\sigma(2)}},\cdots,D^{p_{\sigma(n)}})\\
	&=\mu_n\circ(D^{p_{\sigma(1)}}, D^{p_{\sigma(2)}},\cdots,D^{p_{\sigma(n)}}),\qquad\quad\forall\sigma\in\mathbb{S}_n.\nonumber
\end{align}

Recall that $rNov^{\textup{!}}=Nov$.  By Theorem \ref{Lie-pp}, there exists a morphism of operads from $Lie$ to $rNov\underset{\rm H}{\ot} Nov$, sending $\nu\in Lie(2)$ to
$
e\ot e^*+e^{(1\,2)}\ot (e^{(1\,2)})^*,
$
where $e=\mu\circ(D,\id)\in rNov(2)$. Note that we have
$$
e\ot e^*+e^{(1\,2)}\ot (e^{(1\,2)})^* =e\ot e^*+e^{(1\,2)}\ot (-1)^{(1\,2)}(e^*)^{(1\,2)}=e\ot e^*-e^{(1\,2)}\ot (e^*)^{(1\,2)}.
$$

Let $(\g,\omega)$ be a Novikov algebra. Equivalently, there is a morphism $\beta:Nov\to \End_\g$ of operads sending $e^*$ to $\omega$. Then there is a morphism $\alpha$ of operads from $Lie$ to $rNov\underset{\rm H}{\ot}\End_\g$ as follows:
\begin{eqnarray}\label{Lie-to-rNov-Endg}
&&\alpha:Lie\overset{{\rm Th.}~ \ref{Lie-pp}}{\to} rNov\underset{\rm H}{\ot} rNov^!=rNov\underset{\rm H}{\ot} Nov\overset{\id\ot\beta}{\to}rNov\underset{\rm H}{\ot}\End_\g.
\end{eqnarray}
More precisely, $\alpha$ is given by
\begin{align*}
\alpha(\nu)=\mu\circ(D,\id) \ot \omega-\mu\circ(\id,D)\ot \omega^{(1\,2)}.
\end{align*}
 Thus by Proposition \ref{CEO-cochain}, we obtain the following cochain complex.
\begin{pro}\label{CE-rNov-Endg}
Let $(\g,\omega)$ be a Novikov algebra. Then $(\ceo(rNov\underset{\rm H}{\ot}\End_\g),d_{\rm CE})$ is a cochain complex, where $d_{\rm CE}$ is derived from \eqref{dce} for the above morphism $\alpha$.
\end{pro}
For $n\geqslant 1$, each element of
$\ce^n(rNov\underset{\rm H}{\ot}\End_\g)$ is of the form
$$
\sum\limits_{(p_1,\cdots,p_n)\in P_n} \mu_n\circ(D^{p_1}, D^{p_2},\cdots,D^{p_n})\otimes f^{(p_1,\cdots,p_n)},
$$
where $\mu_n\circ(D^{p_1}, D^{p_2},\cdots,D^{p_n})\in   rNov(n)$, $f^{(p_1,\cdots,p_n)}\in\End_\g(n)$, and satisfies
\begin{eqnarray}\label{CEn-element}
&&\Bigg(\sum_{(p_1,\cdots,p_n)\in P_n} \mu_n\circ(D^{p_1},D^{p_2},\cdots,D^{p_n})\otimes f^{(p_1,\cdots,p_n)}\Bigg)^\sigma\\
&=&(-1)^\sigma\sum_{(p_1,\cdots,p_n)\in P_n} \mu_n\circ(D^{p_1},D^{p_2},\cdots,D^{p_n} )\otimes f^{(p_1,\cdots,p_n)},\nonumber\quad \forall \sigma\in\mathbb{S}_n.
\end{eqnarray}
It is obvious that $(p_1,\cdots,p_n)\in P_n$ if and only if $(p_{\sigma(1)},\cdots,p_{\sigma(n)})\in P_n$ for any $\sigma\in \mathbb{S}_n$. Therefore, for  $\sigma\in\s_n$, we have
\begin{eqnarray*}\label{action-on-Had-prod}
&&\Bigg(\sum_{(p_1,\cdots,p_n)\in P_n} \mu_n\circ(D^{p_1},D^{p_2},\cdots,D^{p_n})\otimes f^{(p_1,\cdots,p_n)}\Bigg)^\sigma\\
&=&\sum_{(p_1,\cdots,p_n)\in P_n} \big(\mu_n\circ(D^{p_1},D^{p_2},\cdots,D^{p_n})\big)^\sigma\otimes (f^{(p_1,\cdots,p_n)})^\sigma\nonumber\\
&\overset{\eqref{action-on-rNov}}=&\sum_{(p_1,\cdots,p_n)\in P_n} \mu_n\circ(D^{p_{\sigma(1)}},D^{p_{\sigma(2)}},\cdots,D^{p_{\sigma(n)}})\otimes (f^{(p_1,\cdots,p_n)})^\sigma,\nonumber
\end{eqnarray*}
and
\begin{eqnarray*}
&&(-1)^\sigma\sum_{(p_1,\cdots,p_n)\in P_n} \mu_n\circ(D^{p_1},D^{p_2},\cdots,D^{p_n} )\otimes f^{(p_1,\cdots,p_n)},\\
&=&(-1)^\sigma\sum_{(p_{\sigma(1)},\cdots,p_{\sigma(n)})\in P_n} \mu_n\circ(D^{p_{\sigma(1)}},D^{p_{\sigma(2)}},\cdots,D^{p_{\sigma(n)}} )\otimes f^{(p_{\sigma(1)},\cdots,p_{\sigma(n)})}\nonumber.
\end{eqnarray*}
Thus, comparing the coefficients of $\mu_n\circ(D^{p_{\sigma(1)}},D^{p_{\sigma(2)}},\cdots,D^{p_{\sigma(n)}})$ on both sides of \eqref{CEn-element}, we obtain
\begin{eqnarray}\label{def-of-scrCn-action}
(f^{(p_1,\cdots,p_n)})^\sigma=(-1)^\sigma f^{(p_{\sigma(1)},\cdots,p_{\sigma(n)})},\quad\forall\sigma\in\mathbb{S}_n,
\end{eqnarray}
which motivates us to define a set $\scrC^n(\g)$ for any $n\geq 1$ by
\begin{eqnarray*}
\scrC^n(\g)=\Bigl\{ \{f^{(p_1,\cdots,p_n)}\}_{(p_1,\cdots,p_n)\in P_n} \;\Big|\; f^{(p_1,\cdots,p_n)}\in\End_\g(n),\
(f^{(p_1,\cdots,p_n)})^\sigma=(-1)^\sigma f^{(p_{\sigma(1)},\cdots,p_{\sigma(n)})} \Bigr\}.
\end{eqnarray*}

For $n\geq 1$, define a linear map $\Psi_n: \ce^n(rNov\underset{\rm H}{\ot}\End_\g) \to \scrC^n(\g)$ by
\begin{eqnarray*}
\Psi_n\Big(\sum_{(p_1,\cdots,p_n)\in P_n} \mu_n\circ(D^{p_1},\cdots,D^{p_n})\otimes f^{(p_1,\cdots,p_n)}\Big)&=&\{f^{(p_1,\cdots,p_n)}\}_{(p_1,\cdots,p_n)\in P_n}.
\end{eqnarray*}

\begin{lem}\label{Psi}
For $n\geq 1$,   $\Psi_n: \ce^n(rNov\underset{\rm H}{\ot}\End_\g)\to \scrC^n(\g)$ is a linear isomorphism of vector spaces.
\end{lem}
\begin{proof}
For any element
$\bigl\{ f^{(p_1,\cdots,p_n)} \bigr\}_{(p_1,\cdots,p_n)\in P_n}\in\scrC^n(\g),$
it is straightforward to check that
$$
F:=\sum_{(p_1,\cdots,p_n)\in P_n} \mu_n\circ(D^{p_1},D^{p_2},\cdots,D^{p_n})\otimes f^{(p_1,\cdots,p_n)}\in\ce^n(rNov\underset{\rm H}{\otimes}\End_\g),
$$
that is, for any $\sigma\in\mathbb{S}_n$,
$F^\sigma=(-1)^\sigma F.$
Thus one can define a linear map from $\scrC^n(\g)$ to $\ce^n(rNov\underset{\rm H}{\ot}\End_\g)$ by
\begin{eqnarray*}
\{f^{(p_1,\cdots,p_n)}\}_{(p_1,\cdots,p_n)\in P_n}&\mapsto&\sum_{(p_1,\cdots,p_n)\in P_n} \mu_n\circ(D^{p_1},\cdots,D^{p_n})\otimes f^{(p_1,\cdots,p_n)},
\end{eqnarray*}
which is obviously the inverse of the map $\Psi_n$ since both compositions yield the identity on the respective spaces. Consequently, $\ce^n(rNov\underset{\rm H}{\ot}\End_\g)$ and $\scrC^n(\g)$ are isomorphic.
\end{proof}

Using the above isomorphism $\Psi_n$, we obtain a cochain complex.
\begin{pro}\label{scrC-cochain}
Let $(\g,\omega)$ be a Novikov algebra. Then $(\scrC^\bullet(\g)=\oplus_{n=1}^{\infty}\scrC^n(\g),\delta)$ is a cochain complex, where the coboundary operator $\delta^n:=\delta|_{\scrC^n(\g)}$ is defined by $\delta^n=\Psi_{n+1}\circ d_{\ce}^n\circ\Psi_n^{-1}$, which can be described as follows:
$$
\xymatrix{
		\text{$\scrC^n(\g)$}
		\ar[rr]^-{ \Psi_{n}^{-1}}
		\ar@{-->}[d]_-{\delta^n}
		&&\text{$\ce^n(rNov\underset{\rm H}{\otimes}\End_\g)$}
		\ar[d]_-{d^n_{\ce}}\\
		\text{$\scrC^{n+1}(\g)$}
		&&\text{$\ce^{n+1}(rNov\underset{\rm H}{\otimes}\End_\g)$}
		\ar[ll]_-{\Psi_{n+1}}}
$$
\end{pro}
\begin{proof}
By Proposition \ref{CEO-cochain}, $(\ceo(rNov\underset{\rm H}{\ot}\End_\g), d_{\rm CE})$ is a cochain complex. Then we have
\begin{align*}
&&\delta^{n+1}\delta^n&=(\Psi_{n+2}\circ d_{\ce}^{n+1}\circ\Psi_{n+1}^{-1})\circ(\Psi_{n+1}\circ d_{\ce}^n\circ\Psi_n^{-1})
=\Psi_{n+2}\circ(d^{n+1}_\ce d^n_\ce)\circ\Psi^{-1}_n
=0.
\end{align*}
Therefore, $\delta^2=0$, and $(\scrC^\bullet(\g)=\oplus_{n=1}^{\infty}\scrC^n(\g),\delta)$ is a cochain complex.
\end{proof}

Before we give an explicit description of the cochain complex of Novikov algebras by use of the cochain complex $(\scrC^\bullet(\g),\delta)$, we introduce two classes of sequences that will be used extensively in the sequel:
\begin{eqnarray}
P^{\ord}_n&=&\{(p_1,p_2,\cdots,p_n)\in P_n\mid p_1\geqslant p_2\geqslant\cdots \geqslant p_n\}\subseteq P_n,\label{P-n-ord}\\
K_n&=&\{(k_{n-1},k_{n-2},\cdots,k_0)\in\dsN^n\mid\sum_{i=0}^{n-1}k_i=n,\,\sum_{i=0}^{n-1}ik_i=n-1\}.\label{K-n}
\end{eqnarray}

For any sequence $(p_1,\cdots,p_n)\in P^{\ord}_n$, we define
\[
k_j = |\{i\mid p_i = j,\,1\leq i\leq n\}|,\quad\forall~0\leq j\leq n-1.
\]
Equivalently, $k_j$ is the multiplicity of the value $j$ in the sequence $(p_1,\cdots,p_n)$. Thus, we obtain
$$\sum_{j=0}^{n-1} k_j = n,\qquad\sum_{j=0}^{n-1} j k_j =\sum_{i=1}^n p_i = n-1,$$
which implies that $(k_{n-1},\cdots,k_0)\in K_n$. Define a map
\begin{equation}\label{phi-n}
	\phi_n:P^{\ord}_n\to K_n
\end{equation}
sending a sequence $(p_1,\cdots,p_n)\in P^{\ord}_n$ to the above sequence $(k_{n-1},\cdots,k_0)\in K_n$.

\begin{lem}\label{P-K}
With the above notations, for any $n\geq 1$,   $\phi_n:P^{\ord}_n \to K_n$ is a bijection.
\end{lem}

\begin{proof}
Given a sequence $(k_{n-1},\cdots,k_0)\in K_n$, since $\sum\limits_{j=0}^{n-1} k_j = n$, we construct a sequence of length $n$ as follows:
\begin{align}\label{map-Kn-to-Pn-ord}
\begin{split}
\left \{
\begin{array}{l}
p_1=\cdots=p_{k_{n-1}}=n-1,\\
p_{k_{n-1}+1}=\cdots=p_{k_{n-1}+k_{n-2}}=n-2,\\
\cdots\\
p_{k_{n-1}+\cdots+k_1+1}=\cdots=p_n=0.
\end{array}
\right.
\end{split}
\end{align}
Then condition $\sum\limits_{j=0}^{n-1} j k_j = n-1$ implies that $\sum\limits_{i=1}^n p_i = n-1$, which shows that $(p_1,\cdots,p_n)\in P_n$. Obviously the  sequence  is non-increasing, so $(p_1,\cdots,p_n)\in P^\ord_n$. Define the map
$$\varphi_n:K_n\to P^{\ord}_n$$
by sending a sequence $(k_{n-1},\cdots,k_0)\in K_n$ to the above sequence $(p_1,\cdots,p_n)\in P^{\ord}_n$.

It is obvious that the map $\varphi_n$ is the inverse of the map $\phi_n$.  Thus, $P^\ord_n$ and $K_n$ are in one-to-one correspondence for any $n\geq 1$.
\end{proof}

We now turn to a more detailed analysis of $P_n$. First, we introduce some notations and lemmas required for this purpose. Define a right action of $\s_n$ on $P_n$ by $$(p_1,\cdots, p_n)^\sigma:=(p_{\sigma(1)},\cdots,p_{\sigma(n)}).$$ In fact, for any $\sigma,\tau\in\s_n$, denote by $q_1=p_{\sigma(1)},\cdots,q_n=p_{\sigma(n)}$, we have
\begin{align*}
	((p_1,\cdots,p_n)^\sigma))^\tau&=(p_{\sigma(1)},\cdots, p_{\sigma(n)})^\tau=(q_1,\cdots,q_n)^\tau=(q_{\theta(1)},\cdots, q_{\theta(n)})\\
	&=(p_{\sigma\tau(1)},\cdots, p_{\sigma\tau(n)})=(p_1,\cdots,p_n)^{\sigma\tau},
\end{align*}
which implies that it is indeed  a right action.

\begin{lem}\label{lem-PP} The set  $P_n$ can be characterized by $P^\ord_n$ as follows:
 	\begin{align*}
		P_n=\{(p_{\tau(1)},\cdots, p_{\tau(n)})\mid(p_1,\cdots,p_n)\in P_n^{\ord},~\tau\in \Sh^{-1}(k_{n-1},\cdots,k_0)=\Sh^{-1}\phi_n(p_1,\cdots,p_n)\}.
	\end{align*}
\end{lem}
\begin{proof}
It is obvious that the right-hand side is contained in the left-hand side. Thus, we only need to show that any element in $P_n$ is of the above form.
 For any $(p_1,\cdots,p_n)\in P_n$, we can reorder the sequence $(p_1,\cdots,p_n)$ such that $p_{i_1}\geq \cdots \geq p_{i_n}$. Then  $(p_{i_1},\cdots,p_{i_n})\in P_{n}^{\ord}$. Let $(k_{n-1},\cdots,k_0)= \phi_n(p_{i_1},\cdots,p_{i_n})$. Define $\tau=\begin{pmatrix}
		1 & 2 & \cdots & n \\
		i_1 & i_2 & \cdots & i_n
\end{pmatrix}$. Then by Lemma \ref{Sh-inverse-decom} \textrm{(ii)}, there are unique permutations $\tau'\in\Sh^{-1}(k_{n-1},\cdots,k_0)$ and $\tau_{n-1}\in \mathbb{S}_{k_{n-1}},\cdots,\tau_0\in \mathbb{S}_{k_0}$ such that $\tau^{-1}=(\tau_{n-1}\times\cdots\times\tau_0)\tau'$. Note that  $(p_{i_1},\cdots,p_{i_n})^{\tau_{n-1}\times\cdots\times\tau_0}=(p_{i_1},\cdots,p_{i_n})$ by \eqref{map-Kn-to-Pn-ord}. Thus, we have
\begin{align*}
(p_1,\cdots,p_n)&=(p_{i_1},\cdots,p_{i_n})^{\tau^{-1}}=(p_{i_1},\cdots,p_{i_n})^{(\tau_{n-1}\times\cdots\times\tau_0)\tau'}\\
&=((p_{i_1},\cdots,p_{i_n})^{(\tau_{n-1}\times\cdots\times\tau_0)})^{\tau'}=(p_{i_1},\cdots,p_{i_n})^{\tau'},
\end{align*}
which finishes the proof.
\end{proof}

Let $(\g,\diamond)$ be a Novikov algebra. For $n\geq 1$, define the space of $n$-cochains by
\begin{equation}\label{eq:Nov-n-cochain}
\frkC_{\nov}^n(\g)=\underset{(k_{n-1},\cdots,k_0)\in K_n}{\oplus}\Hom(\wedge^{k_{n-1}}\g\ot\cdots\ot\wedge^{k_0}\g,\g).
\end{equation}

Obviously we have the following characterization of an $n$-cochain.

\begin{lem}\label{f-act-direct-product}
A multilinear map $f\in \End_\g(n)$ is in $\Hom(\wedge^{k_{n-1}}\g\ot\cdots\ot\wedge^{k_0}\g,\g)$ if and only if for all $\sigma_{n-1}\in \mathbb{S}_{k_{n-1}},\cdots,\sigma_0\in \mathbb{S}_{k_0}$, the following identity holds:
	\begin{align*}
		f^{\sigma_{n-1}\times\cdots\times\sigma_0}=(-1)^{\sigma_{n-1}\times\cdots\times\sigma_0} f,
	\end{align*}
where $(-1)^{\sigma_{n-1}\times\cdots\times\sigma_0}=(-1)^{\sigma_{n-1}}\cdots(-1)^{\sigma_0}$.
\end{lem}

For $n\geq 1$, given an element $\{f^{(p_1,\cdots,p_n)}\}_{(p_1,\cdots,p_n)\in P_n}\in\scrC^n(\g)$,
we claim that
$$
\{f^{(p_1,\cdots,p_n)}\}_{(p_1,\cdots,p_n)\in P_n^{\ord}}\in\frkC_{\nov}^n(\g).
$$
Indeed, given $(p_1,\cdots,p_n)\in P_n^{\ord}$, by the map $\phi_n:P^\ord_n\to K_n$ shown in Lemma \ref{P-K}, we obtain a sequence
$(k_{n-1},\cdots,k_0)=\phi_n(p_1,\cdots,p_n)\in K_n$,
and satisfies the relation \eqref{map-Kn-to-Pn-ord}.
Then for all $\sigma_{n-1}\in \mathbb{S}_{k_{n-1}},\cdots,\sigma_0\in \mathbb{S}_{k_0}$, we have
\begin{eqnarray*}
	(f^{(p_1,\cdots,p_n)})^{\sigma_{n-1}\times\cdots\times\sigma_0}
	&\overset{\eqref{def-of-scrCn-action}}=&(-1)^{\sigma_{n-1}\times\cdots\times\sigma_0}f^{(p_{\sigma_{n-1}(1)},\cdots,p_{\sigma_{n-1}(k_{n-1})},\cdots,p_{k_{n-1}+\cdots+k_1+\sigma_0({1})},\cdots,p_{k_{n-1}+\cdots+k_1+\sigma_0(k_0)})}\\
	&\overset{\eqref{map-Kn-to-Pn-ord}}=&(-1)^{\sigma_{n-1}\times\cdots\times\sigma_0}f^{(p_1,\cdots,p_n)},
\end{eqnarray*}
which implies that $f^{(p_1,\cdots,p_n)}\in\Hom(\wedge^{k_{n-1}}\g\ot\cdots\ot\wedge^{k_0}\g,\g)$ by Lemma \ref{f-act-direct-product}. Hence, we obtain that $\{f^{(p_1,\cdots,p_n)}\}_{(p_1,\cdots,p_n)\in P_n^{\ord}}\in\frkC_{\nov}^n(\g)$. Then define a linear map $\Phi_n:\scrC^n(\g)\to \frkC^n_{\nov}(\g)$ by
\begin{align}
	\{f^{(p_1,\cdots,p_n)}\}_{(p_1,\cdots,p_n)\in P_n}\longmapsto \{f^{(p_1,\cdots,p_n)}\}_{(p_1,\cdots,p_n)\in P_n^{\ord}}.
\end{align}

\begin{pro}\label{psi}
With the above notations, for $n\geq 1$, $\Phi_n:\scrC^n(\g)\to \frkC^n_{\nov}(\g)$ is a linear isomorphism of vector spaces.	
\end{pro}
\begin{proof}
Given $(k_{n-1},\cdots,k_0)\in K_n$ and a map $f_{(k_{n-1},\cdots,k_0)}\in\Hom(\wedge^{k_{n-1}}\g\ot\cdots\ot\wedge^{k_0}\g,\g)\subseteq\frkC_{\nov}^n(\g)$, let $(p_1,\cdots,p_n)=\varphi_n(k_{n-1},\cdots,k_0) \in P_{n}^\ord$ as shown in Lemma \ref{P-K}, and define
\begin{eqnarray*}
f^{(p_1,\cdots,p_n)}&:=&f_{(k_{n-1},\cdots,k_0)},\\
f^{(p_{\sigma(1)},\cdots,p_{\sigma(n)})}&:=&(-1)^\sigma (f^{(p_1,\cdots,p_n)})^\sigma=(-1)^\sigma (f_{(k_{n-1},\cdots,k_0)})^\sigma,\quad \forall \sigma\in \mathbb{S}_n.
\end{eqnarray*}
  We claim that the following set belongs to $\scrC^n(\g)$:
$$X:=\{f^{(p_{\tau(1)},\cdots,p_{\tau(n)})}\mid\tau\in\Sh^{-1}(k_{n-1},\cdots,k_0)\}.$$
Indeed, for any $\sigma\in\s_n$, we have
\begin{align*}
	\big(f^{(p_{\tau(1)},\cdots,p_{\tau(n)})}\big)^\sigma=((-1)^\tau(f^{(p_1,\cdots,p_n)})^\tau)^\sigma =(-1)^\sigma (-1)^{\tau\sigma}(f^{(p_1,\cdots,p_n)})^{\tau\sigma}=(-1)^\sigma f^{(p_{\tau\sigma(1)}, \cdots,p_{\tau\sigma(n)})}.
\end{align*}
By Lemma \ref{lem-PP}, for the element $(p_{\tau\sigma(1)}, \cdots,p_{\tau\sigma(n)})\in P_n$, there exists $\theta\in\Sh^{-1}(k_{n-1},\cdots,k_0)$ such that $(p_{\tau\sigma(1)}, \cdots,p_{\tau\sigma(n)})=(p_{\theta(1)}, \cdots,p_{\theta(n)})$, which implies that $f^{(p_{\tau\sigma(1)}, \cdots,p_{\tau\sigma(n)})}\in X$. Thus, elements of the set $X$ satisfies \eqref{def-of-scrCn-action}, that is, the set $X$ is an element of $\scrC^n(\g)$. Therefore, one can define a linear map from $\frkC^n_{\nov}(\g)$ to $\scrC^n(\g)$ by
\begin{eqnarray*}
f_{(k_{n-1},\cdots,k_0)}&\mapsto&X,
\end{eqnarray*}
which is obviously the inverse of the linear map $\Phi_n$ since both compositions yield the identity on the respective spaces. Consequently, $\scrC^n(\g)$ and  $\frkC^n_{\nov}(\g)$ are isomorphic.
\end{proof}

With these preparations, we are now ready to introduce the cochain complex of Novikov algebras.

\begin{thm}\label{cochain-Nov}
Let $(\g,\diamond)$ be a Novikov algebra. Then $(\frkC_{\nov}^\bullet(\g)=\oplus_{n=1}^{\infty}\frkC_{\nov}^n(\g),d_{\nov})$ is a cochain complex, where the space of $n$-cochains $\frkC_{\nov}^n(\g)$ is given by \eqref{eq:Nov-n-cochain}, and the coboundary operator $d_\nov^n$ is given by
$$d^n_\nov=\Phi_{n+1}\circ \delta^n \circ\Phi_{n}^{-1}=(\Phi_{n+1}\Psi_{n+1})\circ d^n_\ce\circ (\Phi_{n}\Psi_{n})^{-1},$$
which can be described as follows:
	\[
	\xymatrix@C=4em{
		\frkC_{\nov}^n(\mathfrak{g})
		\ar[r]^-{\Phi_{n}^{-1}}
		\ar@{-->}[d]_{d^n_\nov}
		&
		\mathscr{C}^n(\mathfrak{g})
		\ar[r]^-{\Psi_{n}^{-1}}
		\ar[d]_{\delta^n}
		&
		\ce^n(rNov\mathop{\otimes}\limits_{\mathrm{H}}\mathrm{End}_{\mathfrak{g}})
		\ar[d]_{d_\ce^n}
		\\
		\frkC_{\nov}^{n+1}(\mathfrak{g})
		&
		\mathscr{C}^{n+1}(\mathfrak{g})
		\ar[l]_-{\Phi_{n+1}}
		&
		\ce^{n+1}(rNov\mathop{\otimes}\limits_{\rm H}\mathrm{End}_{\mathfrak{g}})
		\ar[l]_-{\Psi_{n+1}}
	}
	\]
\end{thm}
\begin{proof}
By Proposition \ref{scrC-cochain}, $(\scrC^\bullet(\g),\delta)$ is a cochain complex. Then we have
\begin{align*}
d_\nov^{n+1}d_\nov^n&=(\Phi_{n+2}\circ \delta^{n+1} \circ\Phi^{-1}_{n+1})\circ(\Phi_{n+1}\circ \delta^n \circ\Phi_{n}^{-1})\\
&=\Phi_{n+2}\circ(\delta^{n+1}\delta^n)\circ\Phi_{n}^{-1}\\
&=0.
\end{align*}
Therefore, $d_{\nov}^2=0$, and $(\frkC_{\nov}^\bullet(\g)=\oplus_{n=1}^{\infty}\frkC_{\nov}^n(\g),d_{\nov})$ is a cochain complex.
\end{proof}
\begin{defi}\label{def-cohomology-Nov}
Let $(\g,\diamond)$ be a Novikov algebra. The cohomology of the cochain complex $(\frkC_{\nov}^\bullet(\g),d_{\nov})$ is called the {\bf cohomology} of the Novikov algebra $(\g,\diamond)$. We denote the $n$-th cohomology group by $\huaH^n_\nov(\g;\g)$.
\end{defi}
Given $(k_{n-1},\cdots,k_0)\in K_n$ and a map $f_{(k_{n-1},\cdots,k_0)}\in\Hom(\wedge^{k_{n-1}}\g\ot\cdots\ot\wedge^{k_0}\g,\g)\subseteq\frkC_{\nov}^n(\g)$, we now give a detail discussion of $d_\nov(f_{(k_{n-1},\cdots,k_0)})$. By the definition of $d_\nov^n$, we have
\begin{align*}
d_\nov^n(f_{(k_{n-1},\cdots,k_0)})
&=\Big((\Phi_{n+1}\Psi_{n+1})\circ d^n_\ce\circ (\Phi_{n}\Psi_{n})^{-1}\Big)(f_{(k_{n-1},\cdots,k_0)})\\
&=(\Phi_{n+1}\Psi_{n+1})\Big(d_\ce^n\big(\sum_{\sigma\in\Sh^{-1}(k_{n-1},\cdots,k_0)} (-1)^\sigma (\mu_n\circ(D^{p_1},\cdots,D^{p_n})\otimes f_{(k_{n-1},\cdots,k_0)})^\sigma\big)\Big),
\end{align*}
where $(p_1,\cdots,p_n)={\phi_n}^{-1}(k_{n-1},\cdots,k_0)$.
By the definition of $d_\ce^n$ given by \eqref{dce}, we obtain that
\begin{align*}
d_\nov^n(f_{(k_{n-1},\cdots,k_0)})\in
&\quad\Hom(\g\otimes\wedge^{k_{n-1}-1}\g\ot\wedge^{k_{n-2}}\g\ot\cdots\ot\wedge^{k_1}\g\ot\wedge^{{k_0}+1}\g,\g)\\
&\oplus\Hom(\wedge^{k_{n-1}+1}\g\ot\wedge^{k_{n-2}-1}\g\ot\wedge^{k_{n-3}}\g\ot\cdots\ot\wedge^{k_1}\g\ot\wedge^{{k_0}+1}\g,\g)\\
&\oplus\cdots\\
&\oplus\Hom(\wedge^{k_{n-1}}\g\ot\cdots\ot\wedge^{{k_2}+1}\g\ot\wedge^{{k_1}-1}\g\ot\wedge^{{k_0}+1}\g,\g)\\
&\oplus\Hom(\wedge^{k_{n-1}}\g\ot\cdots\ot\wedge^{{k_2}}\g\ot\wedge^{{k_1}+1}\g\ot\wedge^{{k_0}}\g,\g),
\end{align*}
which motivates us to write the operator $d^{(k_{n-1},\cdots,k_0)}:=d_\nov^n|_{\Hom(\wedge^{k_{n-1}}\g\ot\cdots\ot\wedge^{k_0}\g,\g)}$ as follows:
\begin{eqnarray*}
d^{(k_{n-1},\cdots,k_0)}&=&\sum_{(k'_n,k'_{n-1},\cdots,k'_0)\in K'_{n+1}}d^{(k_{n-1},\cdots,k_0)}_{(k'_n,k'_{n-1},\cdots,k'_0)},
\end{eqnarray*}
where $d^{(k_{n-1},\cdots,k_0)}_{(k'_n,k'_{n-1},\cdots,k'_0)}:\Hom(\wedge^{k_{n-1}}\g\ot\cdots\ot\wedge^{k_0}\g,\g)\to\Hom(\wedge^{k'_n}\g\ot\wedge^{k'_{n-1}}\g\ot\cdots\ot\wedge^{k'_0}\g,\g)$ and
\begin{eqnarray*}
K'_{n+1}=
\left\{
\begin{array}{l}
(1,k_{n-1}-1,k_{n-2},\cdots,k_1,k_0+1),\\[2mm]
(0,k_{n-1}+1,k_{n-2}-1,k_{n-3},\cdots,k_1,k_0+1),\\[2mm]
\cdots,\\[2mm]
(0,k_{n-1},\cdots,k_3,k_2+1,k_1-1,k_0+1),\\[2mm]
(0,k_{n-1},\cdots,k_2,k_1+1,k_0)
\end{array}
\right\}.
\end{eqnarray*}
Therefore, the graphical description of $d^{(k_{n-1},\cdots,k_0)}$ is
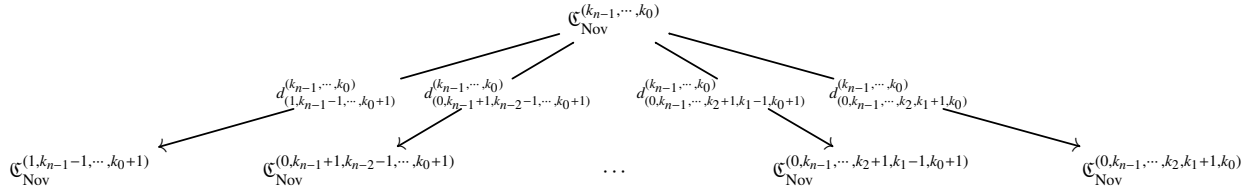
\begin{equation*}
{\tiny
\xymatrix@C=35.7pt@R=39pt{
&&
\frkC_\nov^{(k_{n-1},\cdots,k_0)}
\ar[dll]|{d^{(k_{n-1},\cdots,k_0)}_{(1,k_{n-1}-1,\cdots,k_0+1)}\quad}
\ar[dl]|{\qquad d^{(k_{n-1},\cdots,k_0)}_{(0,k_{n-1}+1,k_{n-2}-1,\cdots,k_0+1)}}
\ar[dr]|{d^{(k_{n-1},\cdots,k_0)}_{(0,k_{n-1},\cdots,k_2+1,k_1-1,k_0+1)}\qquad}
\ar[drr]|{\quad d^{(k_{n-1},\cdots,k_0)}_{(0,k_{n-1},\cdots,k_2,k_1+1,k_0)}}
&&
\\
\frkC_\nov^{(1,k_{n-1}-1,\cdots,k_0+1)}
&\frkC_\nov^{(0,k_{n-1}+1,k_{n-2}-1,\cdots,k_0+1)}
&\cdots
&\frkC_\nov^{(0,k_{n-1},\cdots,k_2+1,k_1-1,k_0+1)}
&\frkC_\nov^{(0,k_{n-1},\cdots,k_2,k_1+1,k_0)}
}
}
\end{equation*}
where $\frkC_\nov^{(k_{n-1},\cdots,k_0)}$ is an abbreviation for $\Hom(\wedge^{k_{n-1}}\g\ot\cdots\ot\wedge^{k_0}\g,\g)$.
We now use the following diagram to illustrate the explicit results of the low-degree differentials.
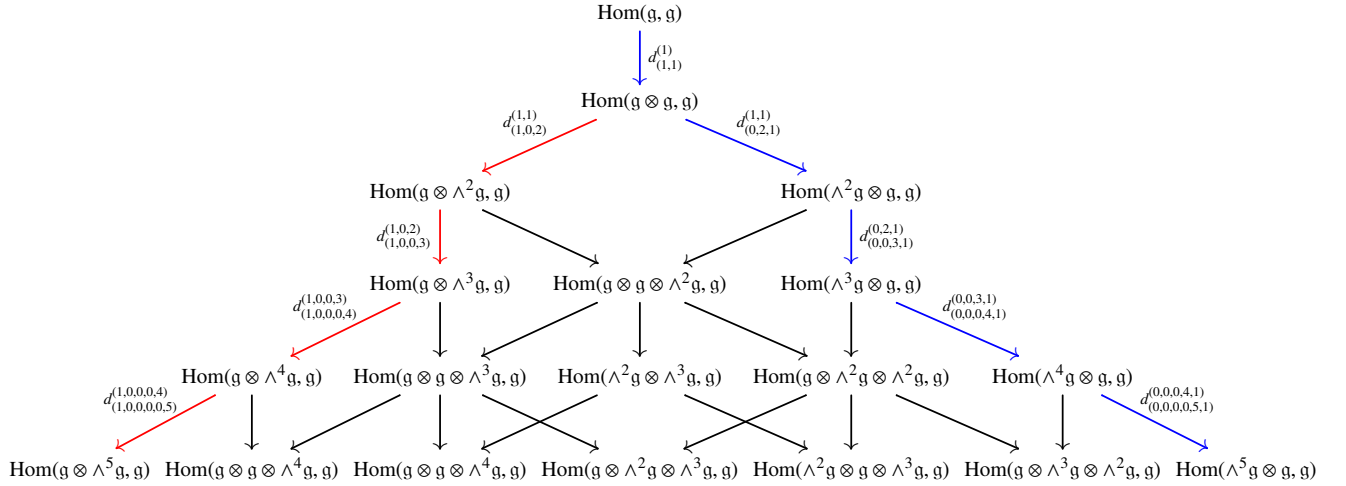
\begin{equation}\label{diagram-Nov-cohomology}
{\tiny
\xymatrix@C=-3.7pt@R=20pt{
&&&&&
&\Hom(\g,\g)
\ar@[blue][d]^-{d^{(1)}_{(1,1)}}
&&&&&
\\
&&&&&
&\Hom(\g\otimes\g,\g)
\ar@[red][dll]_-{d^{(1,1)}_{(1,0,2)}}
\ar@[blue][drr]^-{d^{(1,1)}_{(0,2,1)}}
&&&&&
\\
&&&&\Hom(\g\otimes\wedge^2\g,\g)
\ar@[red][d]_-{d^{(1,0,2)}_{(1,0,0,3)}}
\ar[drr]
&
&
&
&\Hom(\wedge^2\g\otimes\g,\g)
\ar[dll]
\ar@[blue][d]^-{d^{(0,2,1)}_{(0,0,3,1)}}
&&&
\\
&&&
&\Hom(\g\otimes\wedge^3\g,\g)
\ar@[red][dll]_-{d^{(1,0,0,3)}_{(1,0,0,0,4)}}
\ar[d]
&
&\Hom(\g\otimes\g\otimes\wedge^2\g,\g)
\ar[dll]
\ar[d]
\ar[drr]
&
&\Hom(\wedge^3\g\otimes\g,\g)
\ar@[blue][drr]^-{d^{(0,0,3,1)}_{(0,0,0,4,1)}}
\ar[d]
&&&
\\
&
&\Hom(\g\otimes\wedge^4\g,\g)
\ar@[red][dll]_-{d^{(1,0,0,0,4)}_{(1,0,0,0,0,5)}}
\ar[d]
&
&\Hom(\g\otimes\g\otimes\wedge^3\g,\g)
\ar[dll]
\ar[d]
\ar[drr]
&
&\Hom(\wedge^2\g\otimes\wedge^3\g,\g)
\ar[dll]
\ar[drr]
&
&\Hom(\g\otimes\wedge^2\g\otimes\wedge^2\g,\g)
\ar[dll]
\ar[d]
\ar[drr]
&
&\Hom(\wedge^4\g\otimes\g,\g)
\ar@[blue][drr]^-{d^{(0,0,0,4,1)}_{(0,0,0,0,5,1)}}
\ar[d]
&
\\
\Hom(\g\otimes\wedge^5\g,\g)
&
&\Hom(\g\otimes\g\otimes\wedge^4\g,\g)
&
&\Hom(\g\otimes\g\otimes\wedge^4\g,\g)
&
&\Hom(\g\otimes\wedge^2\g\otimes\wedge^3\g,\g)
&
&\Hom(\wedge^2\g\otimes\g\otimes\wedge^3\g,\g)
&
&\Hom(\g\otimes\wedge^3\g\otimes\wedge^2\g,\g)
&
&\Hom(\wedge^5\g\otimes\g,\g)
}
}
\end{equation}

Elements in $\wedge^{k_{n-1}}\g\ot\cdots\ot\wedge^{k_0}\g$ will be written in the form of $x_1,\cdots,x_{k_{n-1}};\cdots;x_{k_{n-1}+\cdots+k_1+1},\cdots,x_{k_{n-1}+\cdots+k_1+k_0}$.

\begin{pro}\label{pro:formular 2-d}
	For any $f_{(1,1)}\in \Hom(\g\ot\g,\g)$, we have $d_{\nov}^2(f_{(1,1)})=   \{d_{(1,0,2)}^{(1,1)}+d_{(0,2,1)}^{(1,1)}\}(f_{(1,1)})$. For all $x_1,x_2,x_3\in \g$, $d_{(1,0,2)}^{(1,1)}(f_{(1,1)})$ and $d_{(0,2,1)}^{(1,1)}(f_{(1,1)})$ are given explicitly by
	\begin{eqnarray*}
		 \big(d_{(1,0,2)}^{(1,1)}\big)(f_{(1,1)})(x_1;x_2,x_3)
		&=&\ff(x_1,x_3)\diamond x_2-\ff(x_1,x_2)\diamond x_3-\ff(x_1\diamond x_2,x_3)+\ff(x_1\diamond x_3,x_2),
	\end{eqnarray*}
	and
	\begin{eqnarray*}
		 \big(d_{(0,2,1)}^{(1,1)}(f_{(1,1)})\big)(x_1,x_2;x_3)
		&=&x_1\diamond\ff(x_2,x_3)-x_2\diamond\ff(x_1,x_3)-\ff(x_1,x_2)\diamond x_3+\ff(x_2,x_1)\diamond x_3\\
		&&-\ff(x_1\diamond x_2,x_3)+\ff(x_2\diamond x_1,x_3)-\ff(x_2,x_1\diamond x_3)+\ff(x_1,x_2\diamond x_3).
	\end{eqnarray*}
\end{pro}
\begin{proof}
	For $n=2$, we have $K_2=\{(1,1)\}$, $P_2=\{(1,0),(0,1)\}$ and $\Sh^{-1}(1,1)=\{(1),(1\,2)\}$. Then $\frkC_{\nov}^2(\g)=\Hom(\g\ot\g,\g)$. Let $f_{(1,1)}\in \Hom(\g\ot\g,\g)$.
 By Theorem \ref{cochain-Nov}, we know that
$$
d(\ff)=\{d_{(0,2,1)}^{(1,1)}(f_{(1,1)}),d_{(2,0,0)}^{(1,1)}(f_{(1,1)})\}\in \Hom(\wedge^2\g\ot\g,\g)\oplus\Hom(\g\ot\wedge^2\g,\g)=\frkC_{\nov}^3(\g).
$$

 By Lemma \ref{Psi}, we have
$$
\Phi_{2}^{-1}(f_{(1,1)})=\{f^{(1,0)}:=f_{(1,1)},f^{(0,1)}:=-(f_{(1,1)})^{(1\,2)}\}.
$$
 By Proposition \ref{psi}, we have
\begin{align*}
	\Psi_2^{-1}\circ\Phi_{2}^{-1}(f_{(1,1)})&=\mu_2\circ(D^1,D^0)\ot f^{(1,0)}+\mu_2\circ(D^0,D^1)\ot f^{(0,1)}\\
	&=\mu_2\circ(D^1,D^0)\ot f_{(1,1)}-\mu_2\circ(D^0,D^1)\ot (f_{(1,1)})^{(1\,2)},
\end{align*}
For simplicity, denote $\omega(x\ot y):=x\diamond y$. Recall that
$$
\alpha(\nu)=(\mu_2\circ(D^1,D^0))\ot \omega-(\mu_2\circ(D^0,D^1))\ot \omega^{(1\,2)},
$$
where actually $D^0=\id$. By abuse of notation, we denote $\id\ot\id$ by $\id_{rNov}\ot\id_{\End_\g}\in rNov(1)\ot \End_\g(1)$.  Thus, we have
\begin{align*}
	&\quad d_\ce^2(\Psi_n^{-1}\circ\Phi_{n}^{-1}(f_{(1,1)}))\\
	&=\alpha(\nu)\circ\bigl(\id\ot \id,(\mu_2\circ(D^1,D^0))\ot f_{(1,1)}-(\mu_2\circ(D^0,D^1))\ot (f_{(1,1)})^{(1\,2)}\bigr)\\
	&\quad -\Big(\alpha(\nu)\circ\bigl(\id\ot \id,(\mu_2\circ(D^1,D^0))\ot f_{(1,1)}-(\mu_2\circ(D^0,D^1))\ot (f_{(1,1)})^{(1\,2)}\bigr)\Big)^{(1\,2)}\\
	&\quad +\Big(\alpha(\nu)\circ\bigl(\id\ot \id,(\mu_2\circ(D^1,D^0))\ot f_{(1,1)}-(\mu_2\circ(D^0,D^1))\ot (f_{(1,1)})^{(1\,2)}\bigr)\Big)^{(1\,2,3)}\\
	&\quad -\bigl(\mu_2\circ(D^1,D^0)\ot f_{(1,1)}-\mu_2\circ(D^0,D^1)\ot (f_{(1,1)})^{(1\,2)}\bigr)\circ (\alpha(\nu),\id\ot\id)\\
	&\quad +\Big(\bigl(\mu_2\circ(D^1,D^0)\ot f_{(1,1)}-\mu_2\circ(D^0,D^1)\ot (f_{(1,1)})^{(1\,2)}\bigr)\circ (\alpha(\nu),\id\ot\id)\Big)^{(2\,3)}\\
	&\quad -\Big(\bigl(\mu_2\circ(D^1,D^0)\ot f_{(1,1)}-\mu_2\circ(D^0,D^1)\ot (f_{(1,1)})^{(1\,2)}\bigr)\circ (\alpha(\nu),\id\ot\id)\Big)^{(1\,3\,2)}.
\end{align*}
There are totally 24 items here. We will only calculate two items in detail and the rest can be obtained similarly. The first item in the first line:
\begin{align*}
	&\quad\big((\mu_2\circ(D^1,D^0))\ot \omega\big)\circ\big(\id\ot \id,(\mu_2\circ(D^1,D^0))\ot f_{(1,1)}\big)\\
	&=\big((\mu_2\circ(D^1,D^0))\circ (\id,\mu_2\circ(D^1,D^0))\big)\ot \big(\omega\circ (\id, f_{(1,1)})\big)\\
	&=\big(\mu_2\circ (D^1\circ\id,D^0\circ(\mu_2\circ(D^1,D^0)))\big)\ot \big(\omega\circ (\id, f_{(1,1)})\big)\\
	&=\big(\mu_2\circ (\id\circ D^1,\mu_2\circ(D^1,D^0))\big)\ot \big(\omega\circ (\id, f_{(1,1)})\big)\\
	&=\big((\mu_2\circ(\id,\mu_2))\circ (D^1,D^1,D^0)\big)\ot \big(\omega\circ (\id, f_{(1,1)})\big)\\
	&=\big(\mu_3\circ(D^1,D^1,D^0)\big)\ot \big(\omega\circ (\id, f_{(1,1)})\big),
\end{align*}
and the first item in the last line:
\begin{align*}
	&\quad\Bigl(\bigl((\mu_2\circ(D^1,D^0)\ot f_{(1,1)}\bigr)\circ\bigl((\mu_2\circ(D^1,D^0))\ot \omega,\id\ot\id\bigr)\Bigr)^{(1\,3\,2)}\\
	&=\Bigl(\bigl((\mu_2\circ(D^1,D^0))\circ (\mu_2\circ(D^1,D^0),\id)\bigr)\ot \bigl(f_{(1,1)}\circ(\omega,\id)\bigr)\Bigr)^{(1\,3\,2)}\\
	&=\Bigl(\bigl(\mu_2\circ (D^1\circ(\mu_2\circ(D^1,D^0)),D^0\circ\id)\bigr)\ot \bigl(f_{(1,1)}\circ(\omega,\id)\bigr)\Bigr)^{(1\,3\,2)}\\
	&=\Bigl(\bigl(\mu_2\circ ((\mu_2\circ(D^1,\id))\circ(D^1,D^0)+(\mu_2\circ(\id,D^1))\circ(D^1,D^0),D^0\circ\id)\bigr)\ot \bigl(f_{(1,1)}\circ(\omega,\id)\bigr)\Bigr)^{(1\,3\,2)}\\
	&=\Bigl(\bigl(\mu_2\circ (\mu_2\circ(D^2,D^0)+\mu_2\circ(D^1,D^1),\id\circ D^0)\bigr)\ot \bigl(f_{(1,1)}\circ(\omega,\id)\bigr)\Bigr)^{(1\,3\,2)}\\
	&=\Bigl(\bigl((\mu_2\circ(\mu_2,\id)) \circ((D^2,D^0,D^0)+(D^1,D^1,D^0))\bigr)\ot \bigl(f_{(1,1)}\circ(\omega,\id)\bigr)\Bigr)^{(1\,3\,2)}\\
	&=\Bigl(\bigl(\mu_3\circ((D^2,D^0,D^0)+(D^1,D^1,D^0))\bigr)\ot \bigl(f_{(1,1)}\circ(\omega,\id)\bigr)\Bigr)^{(1\,3\,2)}\\
	&=\bigl(\mu_3\circ((D^2,D^0,D^0)+(D^1,D^1,D^0))\bigr)^{(1\,3\,2)}\ot \bigl(f_{(1,1)}\circ(\omega,\id)\bigr)^{(1\,3\,2)}\\
	&=\bigl(\mu_3^{(1\,3\,2)}\circ((D^0,D^2,D^0)+(D^0,D^1,D^1))\bigr)\ot \bigl(f_{(1,1)}\circ(\omega,\id)\bigr)^{(1\,3\,2)}\\
	&=\bigl(\mu_3\circ((D^0,D^2,D^0)+(D^0,D^1,D^1))\bigr)\ot \bigl(f_{(1,1)}\circ(\omega,\id)\bigr)^{(1\,3\,2)}.
\end{align*}
Therefore, we have
\begin{align*}
	&\quad d_\ce^2(\Psi_2^{-1}\circ\Phi_2^{-1}(f_{(1,1)}))\\
	&=\bigl(\mu_3\circ(D^1,D^1,D^0)\bigr)\ot \bigl(\omega\circ (\id, f_{(1,1)})\bigr)\\
	&\quad-\bigl(\mu_3\circ((D^0,D^2,D^0)+(D^0,D^1,D^1))\bigr)\ot\bigl(\omega^{(1\,2)}\circ(\id,\ff)\bigr)\\
	&\quad-\bigl(\mu_3\circ(D^1,D^0,D^1)\bigr)\ot \bigl(\omega\circ(\id,(\ff)^{(1\,2)})\bigr)\\
	&\quad+\bigl(\mu_3\circ((D^0,D^1,D^1)+(D^0,D^0,D^2))\bigr)\ot\bigl(\omega^{(1\,2)}\circ(\id,(\ff)^{(1\,2)} )\bigr)\\
	&\quad-\big(\mu_3\circ(D^1,D^1,D^0)\big)\ot \big(\omega\circ (\id, f_{(1,1)})\big)^{(1\,2)}\\
	&\quad+\bigl(\mu_3\circ((D^2,D^0,D^0)+(D^1,D^0,D^1))\bigr)\ot\bigl(\omega^{(1\,2)}\circ(\id,\ff)\bigr)^{(1\,2)}\\
	&\quad+\bigl(\mu_3\circ(D^0,D^1,D^1)\bigr)\ot \bigl(\omega\circ(\id,(\ff)^{(1\,2)})\bigr)^{(1\,2)}\\
	&\quad-\bigl(\mu_3\circ((D^1,D^0,D^1)+(D^0,D^0,D^2))\bigr)\ot\bigl(\omega^{(1\,2)}\circ(\id,(\ff)^{(1\,2)} )\bigr)^{(1\,2)}\\
	&\quad+\bigl(\mu_3\circ(D^1,D^0,D^1)\bigr)\ot \bigl(\omega\circ (\id, f_{(1,1)})\bigr)^{(1\,2\,3)}\\
	&\quad-\bigl(\mu_3\circ((D^2,D^0,D^0)+(D^1,D^1,D^0))\bigr)\ot\bigl(\omega^{(1\,2)}\circ(\id,\ff)\bigr)^{(1\,2\,3)}\\
	&\quad-\bigl(\mu_3\circ(D^0,D^1,D^1)\bigr)\ot \bigl(\omega\circ(\id,(\ff)^{(1\,2)})\bigr)^{(1\,2\,3)}\\
	&\quad+\bigl(\mu_3\circ((D^1,D^1,D^0)+(D^0,D^2,D^0))\bigr)\ot\bigl(\omega^{(1\,2)}\circ(\id,(\ff)^{(1\,2)} )\bigr)^{(1\,2\,3)}\\
	&\quad-\bigl(\mu_3\circ((D^2,D^0,D^0)+(D^1,D^1,D^0))\bigr)\ot\bigl(\ff\circ(\omega,\id)\bigr)\\
	&\quad+\bigl(\mu_3\circ(D^1,D^0,D^1)\bigr)\ot\bigl((\ff)^{(1\,2)}\circ(\omega,\id)\bigr)\\
	&\quad+\bigl(\mu_3\circ((D^1,D^1,D^0)+(D^0,D^2,D^0))\bigr)\ot\bigl(\ff\circ(\omega^{(1\,2)},\id)\bigr)\\
	&\quad-\bigl(\mu_3\circ(D^0,D^1,D^1)\bigr)\ot\bigl((\ff)^{(1\,2)}\circ(\omega^{(1\,2)},\id)\bigr)\\
	&\quad+\bigl(\mu_3\circ((D^2,D^0,D^0)+(D^1,D^0,D^1))\bigr)\ot\bigl(\ff\circ(\omega,\id)\bigr)^{(2\,3)}\\
	&\quad-\bigl(\mu_3\circ(D^1,D^1,D^0)\bigr)\ot\bigl((\ff)^{(1\,2)}\circ(\omega,\id)\bigr)^{(2\,3)}\\
	&\quad-\bigl(\mu_3\circ((D^1,D^0,D^1)+(D^0,D^0,D^2))\bigr)\ot\bigl(\ff\circ(\omega^{(1\,2)},\id)\bigr)^{(2\,3)}\\
	&\quad+\bigl(\mu_3\circ(D^0,D^1,D^1)\bigr)\ot\bigl((\ff)^{(1\,2)}\circ(\omega^{(1\,2)},\id)\bigr)^{(2\,3)}\\
	&\quad-\bigl(\mu_3\circ((D^0,D^2,D^0)+(D^0,D^1,D^1))\bigr)\ot\bigl(\ff\circ(\omega,\id)\bigr)^{(1\,3\,2)}\\
	&\quad+\bigl(\mu_3\circ(D^1,D^1,D^0)\bigr)\ot\bigl((\ff)^{(1\,2)}\circ(\omega,\id)\bigr)^{(1\,3\,2)}\\
	&\quad+\bigl(\mu_3\circ((D^0,D^1,D^1)+(D^0,D^0,D^2))\bigr)\ot\bigl(\ff\circ(\omega^{(1\,2)},\id)\bigr)^{(1\,3\,2)}\\
	&\quad-\bigl(\mu_3\circ(D^1,D^0,D^1)\bigr)\ot\bigl((\ff)^{(1\,2)}\circ(\omega^{(1\,2)},\id)\bigr)^{(1\,3\,2)}.
\end{align*}
Note that $P_3^{\ord}=\{(2,0,0),(1,1,0)\}$. Then comparing the coefficients of $\mu_3\circ(D^2,D^0,D^0)$, we obtain
\begin{align*}
d_{(1,0,2)}^{(1,1)}(f_{(1,1)})=	\bigl(\omega^{(1\,2)}\circ(\id,\ff)\bigr)^{(1\,2)}-\bigl(\omega^{(1\,2)}\circ(\id,\ff)\bigr)^{(1\,2\,3)}-\ff\circ(\omega,\id)+\bigl(\ff\circ(\omega,\id)\bigr)^{(2\,3)}.
\end{align*}
More precisely, for any $x_1,x_2,x_3\in\g$, we have
\begin{align*}
	&\quad d_{(1,0,2)}^{(1,1)}(f_{(1,1)})(x_1;x_2,x_3)\\
	&=\Bigl(\bigl(\omega^{(1\,2)}\circ(\id,\ff)\bigr)^{(1\,2)}-\bigl(\omega^{(1\,2)}\circ(\id,\ff)\bigr)^{(1\,2\,3)}
	-\ff\circ(\omega,\id)+\bigl(\ff\circ(\omega,\id)\bigr)^{(2\,3)}\Bigr)(x_1;x_2,x_3)\\
	&=\ff(x_1,x_3)\diamond x_2-\ff(x_1,x_2)\diamond x_3-\ff(x_1\diamond x_2,x_3)+\ff(x_1\diamond x_3,x_2).
\end{align*}
Similarly, comparing the coefficients of $\mu_3\circ(D^1,D^1,D^0)$, we obtain
\begin{align*}
	d_{(0,2,1)}^{(1,1)}(f_{(1,1)})&=\omega\circ (\id, f_{(1,1)})
	-\big(\omega\circ (\id, f_{(1,1)})\big)^{(1\,2)}
	-\bigl(\omega^{(1\,2)}\circ(\id,\ff)\bigr)^{(1\,2\,3)}
	\\
	&\quad+\bigl(\omega^{(1\,2)}\circ(\id,(\ff)^{(1\,2)} )\bigr)^{(1\,2\,3)}
-\ff\circ(\omega,\id)
	+\ff\circ(\omega^{(1\,2)},\id)
	\\&\quad-\bigl((\ff)^{(1\,2)}\circ(\omega,\id)\bigr)^{(2\,3)}
	+\bigl((\ff)^{(1\,2)}\circ(\omega,\id)\bigr)^{(1\,3\,2)}.
\end{align*}
More precisely, for any $x_1,x_2,x_3\in\g$, we have
\begin{align*}
	&\quad d_{(0,2,1)}^{(1,1)}(f_{(1,1)})(x_1,x_2;x_3)\\
	&=x_1\diamond\ff(x_2,x_3)-x_2\diamond\ff(x_1,x_3)-\ff(x_1,x_2)\diamond x_3+\ff(x_2,x_1)\diamond x_3\\
	&\quad-\ff(x_1\diamond x_2,x_3)+\ff(x_2\diamond x_1,x_3)-\ff(x_2,x_1\diamond x_3)+\ff(x_1,x_2\diamond x_3).
\end{align*}
The proof is finished.
\end{proof}

\begin{rmk}
To make the idea of the proof more observable, one may use the following hint. Given a Novikov algebra $(\g,\diamond)$ and $f_{(1,1)}\in \Hom (\g\otimes \g, \g)$, consider the vector space
$L = F\otimes \g$, where $F$ is the
free commutative associative differential algebras on a countable set of variables $t_1,t_2,\dots $, i.e.,
$F=\mathds K[t_i^{(s)} \mid i=1,2,\dots,\, s=0,1,\dots ]$.
Then $L$ is a Lie algebra with respect to the bracket given by
\[
 [a\otimes x, b\otimes y] = a'b\otimes x\diamond y - ab'\otimes y\diamond x,
\]
for $x,y\in \g$, $a,b\in F$.
Consider
\[
 f(a\otimes x, b\otimes y) = a'b\otimes f_{(1,1)}(x,y) - ab'\otimes f_{(1,1)}(y,x).
\]
This is a skew-symmetric map from $\Hom (L\otimes L, L)$.
Evaluate $(d^2_{\rm CE}f)(t_1\otimes x, t_2\otimes y, t_3\otimes z)$
for $x,y,z \in \g$ by means of the classical formula \eqref{dce}.
Then collect similar terms to get
\[
 (d^2_{\rm CE}f)(t_1\otimes x, t_2\otimes y, t_3\otimes z)
 =
 t_1''t_2t_3 \otimes d^{(1,1)}_{(1,0,2)}(x,y,z) + t_1't_2't_3\otimes d^{(1,1)}_{(0,2,1)}(x,y,z) + \dots ,
\]
where the remaining terms are determined by the skew-symmetry condition.
\end{rmk}

\begin{pro}\label{pro-pL-rcom}
Let $x_1,\dots,x_{n+1}\in\g$.
\begin{itemize}
\item[{\rm (i)}] For $n\geq 1$ and $f\in \Hom(\wedge^{n-1}\g\ot \g,\g)$, the differential
\[
	d_{(0,\cdots,0,n,1)}^{(0,\cdots,0,n-1,1)}:\Hom(\wedge^{n-1}\g\ot \g,\g)\to \Hom(\wedge^{n}\g\ot \g,\g)
\]
is given by
\begin{eqnarray*}
	&&d_{(0,\cdots,0,n,1)}^{(0,\cdots,0,n-1,1)}(f)(x_1,\cdots,x_n;x_{n+1})\\
	&=&\sum\limits_{i=1}^n (-1)^{i+1}x_i \diamond f(x_1,\cdots,\hat{x_i},\cdots;x_{n+1})+\sum\limits_{i=1}^n(-1)^{i+1}f(x_1,\cdots,\hat{x_i},\cdots,x_n;x_i)\diamond x_{n+1}\\
	&&+\sum\limits_{i=1}^n(-1)^{i} f(x_1,\cdots,\hat{x_i},\cdots,x_n;x_i\diamond x_{n+1})\\
	&&+\sum\limits_{1\leq i<j\leq n}(-1)^{i+j} f(x_i\diamond x_j-x_j\diamond x_i,x_1,\cdots,\hat{x_i},\cdots,\hat{x_j},\cdots;x_{n+1}).
\end{eqnarray*}
where  $(0,\cdots,0,n-1,1)\in K_n$ and $(0,\cdots,0,n,1)\in K_{n+1}$.
\item[{\rm (ii)}] 	For $n\geq 2$ and $f\in\Hom(\g\otimes\wedge^{n-1}\g,\g)$,  the differential
\[
d^{(1,0,\dots,0,n-1)}_{(1,0,\dots,0,n)} \colon \Hom(\g\otimes\wedge^{n-1}\g,\g) \to \Hom(\g\otimes\wedge^{n}\g,\g)
\]
is given by
\begin{align*}
	&\bigl(d^{(1,0,\dots,0,n-1)}_{(1,0,\dots,0,n)}(f)\bigr)(x_1;x_2,\dots,x_{n+1}) \\
	&\quad = \sum_{i=2}^{n+1}(-1)^i f(x_1;x_2,\dots,\hat{x_i},\dots,x_{n+1})\diamond x_i
	- \sum_{i=2}^{n+1}(-1)^i f(x_1\diamond x_i;x_2,\dots,\hat{x_i},\dots,x_{n+1}),
\end{align*}
where  $(1,0,\dots,0,n-1)\in K_n$ and $(1,0,\dots,0,n)\in K_{n+1}$. 		
\end{itemize}
\end{pro}
\begin{proof}
	The proof is analogous to that of Proposition~\ref{pro:formular 2-d}.
\end{proof}

For $n\geq 1$, define
\begin{eqnarray*}
	\overline{\frkC}_{\nov}^n(\g):=\underset{(k_{n-1},\cdots,k_0)\in K_n^{\prime}}{\oplus}\Hom(\wedge^{k_{n-1}}\g\ot\cdots\ot\wedge^{k_0}\g,\g),
\end{eqnarray*}
where $K_n^{\prime}=K_n \setminus(0,\cdots,0,n-1,1)$.  By the diagram \eqref{diagram-Nov-cohomology}, $\overline{\frkC}_{\nov}^\bullet(\g)$ is a   subcomplex of $\frkC_{\nov}^\bullet(\g)$, whose $n$-th cohomology group is denoted by $\overline{\huaH}^n_\nov(\g;\g)$.
Thus we have the quotient  complex $\frkC_{\nov}^\bullet(\g)/\overline{\frkC}_{\nov}^\bullet(\g)=\oplus_{i=1}^\infty\Hom(\wedge^{i-1}\g\ot\g,\g)$. By Proposition \ref{pro-pL-rcom}, we obtain the following relation between the cochain complex of a Novikov algebra and the cochain complex of the underlying pre-Lie algebra.
\begin{thm}\label{pro:pre-Lie-cochain}
	Let $(\g,\diamond)$ be a Novikov algebra. Then the cochain complex $\frkC_{\pl}^\bullet(\g)$ of the underlying pre-Lie algebra $(\g,\diamond)$ is isomorphic to a quotient  of the cochain complex  $\frkC_{\nov}^\bullet(\g)$ of the Novikov algebra $(\g,\diamond)$, i.e. we have the following short exact sequence of cochain complexes:
\begin{equation}
0 \to \overline{\frkC}_{\nov}^\bullet(\g) \stackrel{\iota}{\to}  {\frkC}_{\nov}^\bullet(\g) \stackrel{\pi}{\to} \frkC_{\pl}^\bullet(\g) \to 0,
\end{equation}
where $\iota$ and $\pi$ are the natural inclusion and projection.

Consequently, there is a long exact sequence of the cohomology groups:
\begin{equation}
\cdots \to \overline{\huaH}^n_\nov(\g;\g)   \stackrel{ \huaH^n(\iota)}{\to}  \huaH^n_\nov(\g;\g) \stackrel{ \huaH^n(\pi)}{\to} \huaH^n_\pl(\g;\g)   \stackrel{ c^n}{\to}   \overline{\huaH}^{n+1}_\nov(\g;\g) \to\cdots,
\end{equation}
and the connecting map $c^n$ is given by:
\begin{align*}
\begin{split}
c^n([f])= \left \{
\begin{array}{ll}
    {[}0],                    & n=1,\\
    {[}d_{(0,\cdots,0,1,n-2,2)}^{(0,\cdots,0,n-1,1)}(f)],     & n\geq 2,
\end{array}
\right.
\end{split}
\end{align*}
for any $[f]\in \huaH^n_\pl(\g;\g)$.
\end{thm}
\begin{proof}
It is obvious that the short exact sequence of cochain complexes induces a long exact sequence of cohomology groups. Moreover, by diagram chasing, the connecting map $c^n$ is given by:
\begin{align*}
	c^n([f])= [\iota_{n+1}^{-1}( d_\nov^n( \pi_n^{-1}(f)  ) )],
\end{align*}
for any $[f]\in \huaH^n_\pl(\g;\g)$.
Since the connecting map $c^n$ is independent of the choice of the preimage under $\pi_n$, we can choose the natural lift of $f$ itself, namely $\pi_n^{-1}(f)=f$.
Then by the definition of $d_\nov^n$, for any $f\in \Hom(\wedge^{n-1}\g\ot\g,\g)=\frkC_\pl^n(\g)$, we have
\begin{align*}
\begin{split}
d_\nov^n(f)= \left \{
\begin{array}{ll}
	d_{(1,1)}^{(1)}(f),                    & n=1,\\
	\{d_{(0,0,\cdots,0,1,n-2,2)}^{(0,\cdots,0,n-1,1)}(f)+d_{(0,\cdots,0,n,1)}^{(0,\cdots,0,n-1,1)}(f) \},     & n\geq 2.
\end{array}
\right.
\end{split}
\end{align*}
Thus, we obtain
\begin{align*}
	\begin{split}
		\iota_{n+1}^{-1}(d_\nov^n(f))= \left \{
		\begin{array}{ll}
			0,                    & n=1,\\
			d_{(0,0,\cdots,0,1,n-2,2)}^{(0,\cdots,0,n-1,1)}(f),     & n\geq 2.
		\end{array}
		\right.
	\end{split}
\end{align*}
Therefore, one has
\begin{align*}
	\begin{split}
		c^n([f])= \left \{
		\begin{array}{ll}
			{[}0],                    & n=1,\\
			{[}d_{(0,\cdots,0,1,n-2,2)}^{(0,\cdots,0,n-1,1)}(f)],     & n\geq 2.
		\end{array}
		\right.
	\end{split}
\end{align*}
This completes the proof.
\end{proof}

Note that the relation described in Theorem \ref{pro:pre-Lie-cochain} is based on the explicit formula given in Proposition \ref{pro-pL-rcom}. We give an intrinsic explanation of this relation in the following remark.

\begin{rmk}
 It is direct to verify that any perm algebra is a special right Novikov algebra. In particular, there is a surjective morphism of operads $\gamma:rNov\twoheadrightarrow Perm$ defined by
$$
\gamma_n:rNov(n)\twoheadrightarrow Perm(n),\;\mu_n\circ(D^1,\dots,D^1,\underset{\uparrow\atop i\text{-th}}{D^0},D^1,\dots,D^1) \mapsto (\xi_n)^{(n\,n-1\,\cdots\,i)},\quad i=1,\dots,n,
$$
with all other basis elements sent to zero.

Let $(\g,\omega)$ be a Novikov algebra. Then by \eqref{Lie-to-rNov-Endg}, we obtain a morphism of operads $Lie\to  Perm\underset{{\rm H}}{\ot}\End_\g$:
\begin{align*}
	Lie\to rNov\underset{{\rm H}}{\ot}\End_\g \overset{{\rm \gamma\ot\id}}{\to} Perm\underset{{\rm H}}{\ot}\End_\g,\quad \nu\mapsto \xi_2\ot \omega-\xi_2^{(1\,2)}\ot \omega^{(1\,2)}.
\end{align*}

For all $n\geq 1$, the morphism $\gamma_n$ induces a surjective linear map
	\[
	\tilde{\gamma}_n\colon \ce^n(rNov\underset{\rm H}{\ot}\End_\g) \longrightarrow \ce^n(Perm\underset{\rm H}{\ot}\End_\g)
	\]
	defined  by
	\[
	\sum_{\sigma\in\Sh^{-1}(n-1,1)}(-1)^\sigma\bigl(\mu_n\circ(D^1,\dots,D^1,D^0)\ot f\bigr)^{\sigma}
	\longmapsto
	\sum_{\sigma\in\Sh^{-1}(n-1,1)}(-1)^\sigma (\xi_n\ot f)^\sigma,
	\]
where $f\in \Hom(\wedge^{n-1}\g\ot\g,\g)$. It is straightforward to check that
$$\tilde{\gamma}:=\{\tilde{\gamma}_n\}:\ceo(rNov\underset{\rm H}{\ot}\End_\g)\to \ceo(Perm\underset{\rm H}{\ot}\End_\g)
$$
is a morphism of cochain complexes. Hence, we have an isomorphism of cochain complexes:
	\begin{eqnarray*}
		\ceo(rNov\underset{\rm H}{\ot}\End_\g)/\ker\tilde{\gamma} &\cong& \ceo(Perm\underset{\rm H}{\ot}\End_\g).
	\end{eqnarray*}
Therefore by Theorems \ref{pre-Lie-cochain} and \ref{cochain-Nov}, the cochain complex $\frkC_{\pl}^\bullet(\g)$ of the pre-Lie algebra $(\g,\omega)$ is isomorphic to a quotient of the cochain complex  $\frkC_{\nov}^\bullet(\g)$ of the Novikov algebra $(\g,\omega)$.
 \end{rmk}

 The following example shows that there is an essential distinction between the cohomology group of a Novikov algebra and the cohomology group of the underlying pre-Lie algebra.
Based on deformation equations in the infinitesimal deformation, one can define the second-order cohomology group, which is consistent with our definition here. From this perspective, Alhussein computed concrete examples in \cite{Alussein}.

\begin{ex}\label{ex:diff}
Consider the Novikov algebra $(\mathds{K}[t],\diamond)$ given in Example \ref{ex-kt-p} (i) with $\lambda =0$.  Alhussein showed that
	$
	\huaH^2_{\nov}(\mathds{K}[t];\mathds{K}[t])=0
	$ in \cite{Alussein}. Next, we show that $ \huaH^2_\pl(\mathds{K}[t];\mathds{K}[t])\neq 0$.

Define a linear map $\kappa:\mathds{K}[t]\otimes\mathds{K}[t]\to\mathds{K}[t]$ by
	$$\kappa(t^i,t^j)=ijt^{i+j-2},\quad\forall i,j\in\mathds{N}.$$
	A direct computation shows that
	$\kappa \in\ker (d^2_\pl)$.
	We claim that $\kappa\notin\operatorname{Im}d^1_\pl$. Suppose that there exists  $f\in\Hom(\mathds{K}[t],\mathds{K}[t])$ such that
	$\kappa=d^1_\pl(f)$, and write $f(t)=\sum\limits_{n=0}^{\infty}a_nt^n$, i.e.
	$$\kappa(t^i,t^j)=d^1_\pl(f)(t^i,t^j)=t^i\diamond f(t^j)+f(t^i)\diamond t^j-f(t^i\diamond t^j),\quad\forall i,j\in\mathds{N}.$$
	Particularly, we have
	$$
	1=\kappa(t,t)=t\diamond f(t)+f(t)\diamond t-f(t\diamond t)=\sum\limits_{n=1}^{\infty}na_nt^n\in t\mathds K[t],
	$$
	which is obviously impossible. Therefore, $\kappa\notin\operatorname{Im}d^1_\pl$. Consequently,
	$
	0\neq[\kappa]\in H^2_\pl(\mathds{K}[t];\mathds{K}[t]),
	$
	so $ \huaH^2_\pl(\mathds{K}[t];\mathds{K}[t])\neq 0$. 	
	Therefore, the cohomology group of a Novikov algebra and the cohomology group of the underlying pre-Lie algebra are quite different.
\end{ex}

\section{Infinitesimal deformations of Novikov algebras}\label{sec:def}

In this section, we study infinitesimal deformations of Novikov algebras using the cohomology theory developed in the previous section.

\begin{defi}
	Let $(\g, \diamond)$ be a Novikov algebra and $\ome \in \Hom(\g\ot\g, \g)$. Define a binary operation on $\mathds{K}[t]/(t^2)\ot\g$ by
	$$
	x\diamond_t y:=x\diamond y+t \ome(x,y),\quad\forall x,y\in\g.
	$$
	If $(\mathds{K}[t]/(t^2)\ot\g, \diamond_t)$ is a Novikov algebra, we say that $\ome$ generates an \textbf{infinitesimal deformation} of $(\g, \diamond)$.
\end{defi}

\begin{defi}
Let $(\mathds{K}[t]/(t^2)\ot\g, \diamond_t)$, $(\mathds{K}[t]/(t^2)\ot\g, \diamond_t^{\prime})$ be two infinitesimal deformations of a Novikov algebra $(\g,\diamond)$ generated by $\ome$ and $\ome^{\prime}$ respectively.  They are called \textbf{equivalent} if there is a linear map $N\in\End(\g)$ such that $\id+tN:(\g,\diamond_t)\to (\g,\diamond_t^{\prime})$ is a homomorphism of Novikov algebras.
\end{defi}

\begin{thm}\label{deformation-2-cocycle}
	Let $(\g, \diamond)$ be a Novikov algebra.
\begin{itemize}
\item[{\rm (i)}]  $\ome\in\Hom(\g\otimes\g,\g)$ generates an infinitesimal deformation of $\g$ if and only if $\ome$ is a $2$-cocycle of $\g$.
\item[{\rm (ii)}] Two infinitesimal deformations $(\mathds{K}[t]/(t^2)\ot\g, \diamond_t)$, $(\mathds{K}[t]/(t^2)\ot\g, \diamond_t^{\prime})$ of the Novikov algebra $(\g,\diamond)$ generated by $\ome$ and $\ome^{\prime}$ respectively are equivalent if and only if  $\ome$ and $\ome^{\prime}$ are in the same cohomology class in  $\huaH^2(\g;\g)$.
\end{itemize}	
\end{thm}
\begin{proof}
{\rm (i)} Applying $\diamond_t$ to \eqref{l-Nov-1} and \eqref{l-Nov-2}, for any $x,y,z\in\g$, we have
\begin{align*}
0&=(x\diamond_t y)\diamond_t z-x\diamond_t (y\diamond_t z)-(y\diamond_t x)\diamond_t z+y\diamond_t (x\diamond_t z)\\
&= \bigl((x\diamond y)\diamond z-x\diamond(y\diamond z)-(y\diamond x)\diamond z+y\diamond(x\diamond z)\bigr) \\
&\quad + t\Bigl(\omega(x\diamond y,z)+\omega(x,y)\diamond z-\omega(x,y\diamond z)-x\diamond\omega(y,z)\\
&\qquad\quad-\omega(y,x)\diamond z-\omega(y\diamond x,z)+y\diamond\omega(x,z)+\omega(y,x\diamond z)\Bigr) \\
&\quad + t^{2}\Bigl(\omega(\omega(x,y),z)-\omega(x,\omega(y,z))-\omega(\omega(y,x),z)+\omega(y,\omega(x,z))\Bigr),
\end{align*}
and
\begin{align*}
0&=(x\diamond_t y)\diamond_t z-(x\diamond_t z)\diamond_t y\\
&=\bigl((x\diamond y)\diamond z-(x\diamond z)\diamond y\bigr) + t\Bigl(\omega(x,y)\diamond z+\omega(x\diamond y,z)-\omega(x,z)\diamond y-\omega(x\diamond z,y)\Bigr) \\
&\quad + t^{2}\Bigl(\omega(\omega(x,y),z)-\omega(\omega(x,z),y)\Bigr).
\end{align*}
Since $(\g,\diamond)$ is a Novikov algebra and $t^2=0$ in $\mathds{K}[t]/(t^2)$, the operation $\diamond_t$ satisfies \eqref{l-Nov-1} and \eqref{l-Nov-2} if and only if
\begin{eqnarray}
	0&=&x\diamond\ome(y,z)+\ome(x,y\diamond z)-\ome(x,y)\diamond z-\ome(x\diamond y,z)\label{2-cocycle1}\\
	&&-y\diamond\ome(x,z)-\ome(y,x\diamond z)+\ome(y,x)\diamond z+\ome(y\diamond x,z),\nonumber\\
	0&=&\ome(x,z)\diamond y+\ome(x\diamond z,y)-\ome(x,y)\diamond z-\ome(x\diamond y,z).\label{2-cocycle2}
\end{eqnarray}
By Proposition \ref{pro:formular 2-d}, \eqref{2-cocycle1} and \eqref{2-cocycle2} hold if and only if $d_\nov^{2}(\ome)=0$, that is, $\ome$ is a $2$-cocycle of the Novikov algebra $(\g,\diamond)$.

\text{(ii)}
For any $x,y\in\g$, we have
\begin{align*}
&\quad(\mathrm{id}+tN)(x\diamond_t y)-(\mathrm{id}+tN)(x)\diamond_t'(\mathrm{id}+tN)(y)\\
&=x\diamond y-x\diamond y+t\Bigl(\omega(x,y)+N(x\diamond y)-x\diamond N(y)-N(x)\diamond y-\omega'(x,y)\Bigr)\\
&\quad+t^2\Bigl(N(\omega(x,y))-N(x)\diamond N(y)-\omega'(x,N(y))-\omega'(N(x),y)\Bigr)-t^3\omega'(N(x),N(y)).
\end{align*}
Since we are working over $\mathds{K}[t]/(t^2)$, all terms involving
$t^2$ or higher powers vanish.
Therefore, $\mathrm{id}+tN$ is a homomorphism of Novikov algebras if
and only if
\[
\omega(x,y)-\omega'(x,y)
=
N(x)\diamond y+x\diamond N(y)-N(x\diamond y)=d_\nov^1(N)(x,y),
\qquad \forall\,x,y\in\g.
\]
Hence, $\omega$ and $\omega'$ represent the same cohomology class in
$\huaH^2(\g;\g)$ if and only if the corresponding infinitesimal
deformations of $(\g,\diamond)$ are equivalent.
\end{proof}

\begin{defi}
Let $(\g,\diamond)$ be a Novikov algebra.  An infinitesimal deformation $(\mathds{K}[t]/(t^2)\ot\g, \diamond_t)$  generated by $\ome$ is called \textbf{trivial} if there exists $N\in\End(\g)$ such that
$\id+tN:(\g,\diamond_t)\to (\g,\diamond)$ is an isomorphism of Novikov algebras.
\end{defi}

\begin{defi}
	A Novikov algebra $(\g,\diamond)$ is called \textbf{rigid} if every
	infinitesimal deformation of $(\g,\diamond)$ is trivial.
\end{defi}

\begin{pro}
Let $(\g,\diamond)$ be a Novikov algebra. If $\huaH^2(\g;\g)=0$, then $(\g,\diamond)$ is rigid.
\end{pro}
\begin{proof}
	Let
$(\mathds{K}[t]/(t^2)\ot\g, \diamond_t)$ be an infinitesimal deformation of the Novikov algebra $(\g,\diamond)$ generated by $\ome$. Then by Theorem \ref{deformation-2-cocycle}, $\omega$ is a $2$-cocycle. If $\huaH^2(\g;\g)=0$, then there exists $N\in\End(\g)$ such that $\ome=d_\nov^1(N)$.
That is, for all $x,y\in\g$, we have
$$
	\ome(x,y)=N(x)\diamond y+x\diamond N(y)-N(x\diamond y).
$$
Hence, modulo $t^2$, we have
	\begin{align*}
		(\id+tN)(x\diamond_t y)
		&=(\id+tN)\bigl(x\diamond y+t\ome(x,y)\bigr)\\
		&=x\diamond y+t\bigl(\ome(x,y)+N(x\diamond y)\bigr)\\
		&=x\diamond y+t\bigl(N(x)\diamond y+x\diamond N(y)\bigr)\\
		&=(x+tN(x))\diamond(y+tN(y))\\
		&=(\id+tN)(x)\diamond(\id+tN)(y).
	\end{align*}
Therefore, $\id+tN:(\g,\diamond_t)\to(\g,\diamond)$ is a homomorphism of Novikov algebras. Moreover, its inverse is $\id-tN$ modulo $t^2$. Hence, $\id+tN$ is an isomorphism of Novikov algebras, that is, the infinitesimal deformation generated by $\omega$ is trivial. Therefore, $(\g,\diamond)$ is rigid.
\end{proof}

\begin{rmk}\label{rmk-Alussein}
In \cite{Alussein}, Alusssein showed that
   $\huaH^2(\g;\g)=0$ for the Novikov algebra $(\g=\mathds{K}[t],\diamond)$ given in Example \ref{ex-kt-p} (i) with $\lambda =0$;   while $\huaH^2(\g;\g)\neq 0$ for the Novikov algebra $(\mathds{K}[t]/(t^p),\diamond)$ given in Example \ref{ex-kt-p} (ii). Therefore, the Novikov algebra $(\g=\mathds{K}[t],\diamond)$ is rigid and  $(\mathds{K}[t]/(t^p),\diamond)$ is not.

   On the other hand, according to Example \ref{ex:diff}, we have
  $\dim\huaH^2_\pl(\mathds{K}[t];\mathds{K}[t])\geq 1$, which implies that the underlying pre-Lie algebra of the Novikov algebra $(\g=\mathds{K}[t],\diamond)$ given in Example \ref{ex-kt-p} (i) with $\lambda =0$ is not cohomologically rigid.
\end{rmk}

\section{Cohomology of Novikov algebras with coefficients in representations}\label{sec:rep}

In this section, we   use pseudo-tensor categories to investigate the cohomology theory of Novikov algebras with coefficients in arbitrary representations. We first recall some basic notions.
A pseudo-tensor category is a category equipped with ``polylinear map'' and a way to compose them.
\begin{defi}\cite{BDK,BKV,BD}
A \textbf{pseudo-tensor category} consists of
\begin{itemize}
  \item a class of objects $\frkM$;
  \item a collection of vector spaces $\lin_I^\frkM(\{L_i\}_{i \in I}, M)$ on which the symmetric group $S_I$ acts, where $I$ is a finite non-empty set, $\{L_i\}_{i\in I}$ is a family of objects and $M$ is an object; its elements are called {\bf $I$-operations};
  \item the composition map for any surjection of finite non-empty sets $J\overset{\pi}\twoheadrightarrow I$ and families of objects $\{L_i\}_{i\in I}$, $\{K_j\}_{j\in J}$ and $M$:
  \begin{eqnarray}\label{ps-cate}
  \qquad\lin_I^\frkM(\{L_i\}_{i \in I},M)\times\prod_{I}\lin_{J_i}^\frkM(\{K_j\}_{j \in {J_i}},L_i)\to\lin_J^\frkM(\{K_j\}_{j \in J},M),\quad(\phi,\{\psi_i\}_{i\in I})\mapsto\phi\cc(\{\psi_i\}_{i\in I}),
  \end{eqnarray}
  where $J_i=\pi^{-1}(i)\subset J$ for any $i\in I$,
\end{itemize}
such that the following properties hold:
\begin{itemize}
  \item \textbf{Associativity}: If $H\twoheadrightarrow J$ is another surjection between finite non-empty sets, $\{F_{h}\}_{h\in H}$ is a family of objects and $\varphi_j\in\lin_{H_j}^\frkM(\{F_h\}_{h\in {H_j}},K_j)$, then
      $$\phi\cc\big(\{\psi_i\cc(\{\varphi_j\}_{j\in J})\}_{i\in I}\big)=\big(\phi\cc(\{\psi_i\}_{i\in I})\big)\cc(\{\varphi_j\}_{j\in J})\in\lin_H^\frkM(\{F_h\}_{h\in H}, M).$$
  \item \textbf{Unit}: For any object $M$, there is an element $\id_M \in \lin_{[1]}^\frkM(\{M\}, M)$ such that for any  $I$-operation $\phi\in \lin_I^\frkM(\{L_i\}_{i \in I}, M)$, $$\id_M\cc \phi = \phi\cc(\{\id_{L_i}\}_{i \in I}) = \phi.$$
  \item \textbf{Equivariance}: The compositions \eqref{ps-cate} are equivariant with respect to the natural action of the symmetric group.
\end{itemize}
\end{defi}

In the sequel, we denote a pseudo-tensor category by its object class $\frkM$.

\begin{ex}
An operad is a pseudo-tensor category with only one object. Thus the notion of pseudo-tensor category is a generalization of operad.
\end{ex}

\begin{ex}\label{ps-cat-vec}
Denote by $\huaV ec$ the class of vector spaces. For a finite non-empty set $I$, a family of vector spaces $\{L_i\}_{i \in I}$ and a vector space $M$, we define $$
\lin_I^{\huaV ec}(\{L_i\}_{i \in I},M)=\Hom(\otimes_{i \in I} L_i, M).
$$
The symmetric group $S_I$ acts on the above vector space by permuting the factors in $\otimes_{i \in I} L_i$.
For any surjection of finite non-empty sets $J \overset{\pi}\twoheadrightarrow I$ and a family of objects $\{K_j\}_{j \in J}$, the composition map is defined by
\begin{eqnarray}
\Hom(\otimes_{i \in I}L_i,M)\otimes \prod_{I} \Hom(\otimes_{j \in J_i}{K_j},L_i) &\to& \Hom(\otimes_{j \in J}{K_j},M),\\
(f,\{g_i\}_{i \in I}) &\mapsto&f\cc(\{g_i\}_{i \in I}):=f (\otimes_{i \in I} g_i),\nonumber
\end{eqnarray}
where $J_i = \pi^{-1}(i)$ for any $i \in I$. Therefore, $\huaV ec$ is a pseudo-tensor category.
\end{ex}

\begin{ex}\label{operad-ot-pscategory}
	Let $\mathcal{O}$ be an operad and $\frkM$ a pseudo-tensor category.
Then $\mathcal{O}\otimes\frkM$ is a pseudo-tensor category
	whose objects are those of $\frkM$ and whose multilinear morphism spaces
	are defined as follows. For a finite non-empty set $I$, a family of objects
	$\{L_i\}_{i\in I}$, and an object $M$, let $n=|I|$ and set
	$$
	\lin_{I}^{\mathcal{O}\otimes\frkM}
	(\{L_i\}_{i\in I},M)
	:=
	\mathcal{O}(n)\otimes
	\lin_{I}^{\frkM}(\{L_i\}_{i\in I},M).
	$$
For a surjection $J\twoheadrightarrow I$, let
	$J_i:=\pi^{-1}(i)$ for each $i\in I$. For $f\in\mathcal{O}(n),
	f_i\in\mathcal{O}(|J_i|),
	g\in\lin_{I}^{\frkM}(\{L_i\}_{i\in I},M),$ and $g_i\in
	\lin_{J_i}^{\frkM}(\{K_j\}_{j\in J_i},L_i),
	$
	 the composition is given by
\begin{align*}
	\lin_{I}^{\mathcal{O}\ot \frkM}(\{L_i\}_{i\in I},M)\times \prod_{I}\lin_{J_i}^{\mathcal{O}\ot \frkM}(\{K_j\}_{j \in {J_i}},L_i) &\to \lin_J^{\mathcal{O}\ot \frkM}(\{K_j\}_{j \in J},M),\nonumber\\
	\big(f\ot g, \{f_i\ot g_i\}_{i\in I}\big) &\mapsto (f\ot g)\circ (\{f_i\ot g_i\}_{i\in I}) := \bigl(f\circ\{f_i\}_{i\in I}\bigr) \otimes \bigl(g\circ \{g_i\}_{i\in I}\bigr).
\end{align*}
\end{ex}

\begin{defi}\cite{BD}
 Let $\frkM$ and $\frkN$ be two pseudo-tensor categories. $\alpha : \frkN \to \frkM$ is called a \textbf{pseudo-tensor functor} if it sends any object $N \in \frkN$ to an object $\alpha(N) \in \frkM$, and assigns any finite non-empty set $I$ and objects $\{N_i\}_{i \in I}, N\in \frkN$ to a map $\alpha_I : \lin_I^{\frkN}(\{N_i\}_{i\in I}, N) \to \lin_I^{\frkM}(\{\alpha(N_i)\}_{i\in I}, \alpha(N))$ such that $\alpha_I$ are compatible with the composition map and $\alpha(\mathrm{id}_N) = \mathrm{id}_{\alpha(N)}$. One defines a morphism between pseudo-tensor functors in the obvious way.
\end{defi}

\begin{ex}
Let $\huaP$ and $\huaP'$ be two operads. Then any morphism of operads $\huaP\to\huaP'$ is a pseudo-tensor functor from $\huaP$ to $\huaP'$.
\end{ex}

The language of pseudo-tensor functors provides a natural way to define algebras over an operad in a pseudo-tensor category.

\begin{defi}\cite{BD}\label{operad-alg-in-pt-cat}	
Let $\huaP$ be an operad and $\frkM$ a pseudo-tensor category. An {\bf$\huaP$-algebra} in $\frkM$ is a pseudo-tensor functor $\huaP\to \frkM$.
\end{defi}

\begin{ex}\cite{BDK,BKV}
A \textbf{Lie algebra} in a pseudo-tensor category $\frkM$ consists of an object $\huaL\in\frkM$ and a map $\nu \in \lin^\frkM_{[2]}(\{\huaL,\huaL\},\huaL)$ satisfying:
\begin{itemize}
  \item {Skew-symmetry}: $\nu = -\nu^{(1\,2)}$;
  \item {Jacobi identity}: $\nu\cc(\nu, \id)+(\nu\cc(\id, \nu))^{(1\,2\,3)}+(\nu\cc(\id, \nu))^{(1\,3\,2)}=0.$
\end{itemize}
\end{ex}

\begin{defi}\cite{BDK,BKV}
A \textbf{representation of a Lie algebra} $(\huaL, \nu)$ in a pseudo-tensor category $\frkM$ is an object $M$ together with
$\rho \in \lin_{[2]}^\frkM(\{\huaL, M\}, M)$ satisfying
$$
\rho\cc(\nu,\id_M) = \rho\cc(\id_\huaL, \rho) - (\rho\cc(\id_\huaL,\rho))^{(1\,2)}.
$$
\end{defi}

Note that any  $\mathcal{P}$-algebra in the pseudo-tensor category $\mathcal{V}ec$ is just the ordinary $\mathcal{P}$-algebra. In particular, a Novikov algebra in the pseudo-tensor category $\mathcal{V}ec$ is equivalent to the one given by Definition \ref{def-l-r-Nov}. A \textbf{representation} $(M;l,r)$ of a Novikov algebra $(\g, \omega)$ in the pseudo-tensor category $\mathcal{V}ec$ is the same as the one given in \cite{O1995}. More precisely, $M\in \mathcal{V}ec$ and $l,r\in \Hom(\g\ot M,M)$ satisfy
\begin{align}
	l\cc(\omega,\id_M)-(l\cc(\omega,\id_M))^{(1\,2)}&=l\cc(\id_\g, l) - (l\cc(\id_\g, l))^{(1\,2)},\label{Nov-rep-1}\\
	l\cc(\id_\g, r) - (r\cc(\id_\g, l))^{(1\,2)}&=r\cc(\omega,\id_M)-(r\cc(\id_\g, r))^{(1\,2)},\label{Nov-rep-2}\\
	l\cc(\omega,\id_M)&=(r\cc(\id_\g, l))^{(1\,2)},\label{Nov-rep-3}\\
	r\cc(\id_\g, r)&=(r\cc(\id_\g, r))^{(1\,2)}.\label{Nov-rep-4}
\end{align}

One can define an analogue of Chevalley-Eilenberg cohomology within the framework of pseudo-tensor categories.

\begin{defi}\cite{BKV}\label{CE-comp-wrt-mod}
Let $\frkM$ be a pseudo-tensor category, $(\huaL,\nu)$ a Lie algebra in $\frkM$ and $(M;\rho)$ a representation of $(\huaL,\nu)$.
The space of {\bf $n$-cochains} in $\frkM$ of $(\huaL,\nu)$ with coefficients in $(M;\rho)$ is $$\ce_\frkM^n(\huaL;M)=\{f\in\lin_{[n]}^\frkM(\{\huaL,\cdots,\huaL\},M)\mid f^\sigma=(-1)^\sigma f,  \forall \sigma\in \s_n\}.$$
The coboundary map is defined as follows:
\begin{eqnarray}\label{dce-coeff}
 	&&\hat{d}_{\rm CE}^n(f)=\underset{\sigma\in\Sh^{-1}(1,n)}{\sum}(-1)^{\sigma}\big(\rho\cc(\id_\huaL,f)\big)^{\sigma}-\underset{\sigma\in\Sh^{-1}(2,n-1)}{\sum}(-1)^{\sigma}\big(f\cc(\nu,\id_\huaL,\cdots,\id_\huaL)\big)^{\sigma}.
 \end{eqnarray}
\end{defi}

It is obvious that $(\ce^\bullet_{\huaV ec}(\huaL;M),\hat{d}_\ce)$ is the Chevalley-Eilenberg cochain complex of the Lie algebra $(\huaL,\nu)$ with coefficients in $(M;\rho)$, where $\huaV ec$ is the pseudo-tensor category shown in Example \ref{ps-cat-vec}.

Next we focus on the pseudo-tensor category $rNov\ot \mathcal{V}ec$. Let $(\g,\omega)$ be a Novikov algebra in the pseudo-tensor category $\mathcal{V}ec$, that is, there is a pseudo-tensor functor $\beta:Nov\to \mathcal{V}ec$. Then there is a pseudo-tensor functor $\alpha:Lie\to rNov\ot\mathcal{V}ec$ as follows:
\begin{align*}
\alpha:	Lie\overset{{\rm Th.}~ \ref{Lie-pp}}{\to} rNov\underset{\rm H}{\ot} Nov\overset{\id\ot\beta}{\to} rNov\ot\mathcal{V}ec.
\end{align*}
More precisely, $\alpha$ is given by
\begin{align}\label{def-coeff-nu}
\alpha(\nu)=\mu\cc(D,\id)\ot \omega-\mu\cc(\id, D)\ot\omega^{(1\,2)}\in rNov(2)\otimes \Hom(\g\ot\g,\g),
\end{align}
where $\id$ denotes  $\id_{rNov}$.
Equivalently, $(\g,\alpha(\nu))$ is a Lie algebra in the pseudo-tensor category $rNov\otimes\mathcal{V}ec$.
\begin{lem}\label{Nov-to-Lie}
With the above notations. Let $M$ a vector space and $l,r\in \Hom(\g\ot M, M)$. Define a map $\rho\in\Hom_{[2]}^{rNov\ot\mathcal{V}ec}(\{\g,M\},M)=rNov(2)\ot\Hom(\g\ot M,M)$  by
\begin{align}\label{def-rho}
	\rho=\mu\cc(D,\id)\ot l-\mu\cc(\id,D)\ot r.
\end{align}
Then $(M;\rho)$ is a representation of the Lie algebra $(\g,\alpha(\nu))$ in the pseudo-tensor category $rNov\ot\mathcal{V}ec$ if and only if $(M;l,r)$ is a representation of the Novikov algebra $(\g,\omega)$.
\end{lem}

\begin{proof}
 By \eqref{def-rho}, in the pseudo-tensor category $rNov\ot\mathcal{V}ec$, we have
	\begin{align*}
	&\rho\cc(\alpha(\nu),\id\ot\id_M)-\rho\cc(\id\ot\id_\g, \rho)+(\rho\cc(\id\ot\id_\g,\rho))^{(1\,2)}\\
	&=(\mu\cc(D,\id)\ot l-\mu\cc(\id,D)\ot r)\cc(\mu\cc(D,\id)\ot \omega-\mu\cc(\id, D)\ot\omega^{(1\,2)},\id\ot\id_M)\\
	&\quad-(\mu\cc(D,\id)\ot l-\mu\cc(\id,D)\ot r)\cc(\id\ot\id_\g,\mu\cc(D,\id)\ot l-\mu\cc(\id,D)\ot r)\\
	&\quad+\bigl((\mu\cc(D,\id)\ot l-\mu\cc(\id,D)\ot r)\cc(\id\ot\id_\g,\mu\cc(D,\id)\ot l-\mu\cc(\id,D)\ot r)\bigr)^{(1\,2)}\\
	&=\mu_3\cc((D^2,\id)+(D,D),\id)\ot l\cc(\omega,\id_M)-\mu_3\cc(D,\id,D)\ot r\cc(\omega,\id_M)\\
	&\quad-\mu_3\cc((D,D)+(\id,D^2),\id)\ot l\cc(\omega^{(1\,2)},\id_M)+\mu_3\cc(\id,D,D)\ot r\cc(\omega^{(1\,2)},\id_M)\\
	&\quad-\mu_3\cc(D,D,\id)\ot l\cc(\id_\g,l)+\mu_3\cc(\id,(D,D)+(D^2,\id))\ot r\cc(\id_\g,l)\\
	&\quad+\mu_3\cc(D,\id,D)\ot l\cc(\id_\g,r)-\mu_3\cc(\id,(D,D)+(\id,D^2))\ot r\cc(\id_\g,r) \\
	&\quad+\mu_3\cc(D,D,\id)\ot (l\cc(\id_\g,l))^{(1\,2)}\\
	&\quad-\mu_3\cc(D,\id,D)\ot (r\cc(\id_\g,l))^{(1\,2)}-\mu_3\cc(D^2,\id,\id)\ot (r\cc(\id_\g,l))^{(1\,2)}\\
	&\quad-\mu_3\cc(\id,D,D)\ot (l\cc(\id_\g,r))^{(1\,2)}\\
	&\quad+\mu_3\cc(D,\id,D)\ot (r\cc(\id_\g,r))^{(1\,2)}+\mu_3\cc(\id,\id,D^2)\ot (r\cc(\id_\g,r))^{(1\,2)}\\
	&=\mu_3\cc(D^2,\id,\id)\ot\bigl(l\cc(\omega,\id_M)-(r\cc(\id_\g,l))^{(1\,2)}\bigr)\\
	&\quad+\mu_3\cc(\id,D^2,\id)\ot\bigl(-l\cc(\omega^{(1\,2)},\id_M)+ r\cc(\id_\g,l)\bigr)\\
	&\quad+\mu_3\cc(\id,\id,D^2)\ot\bigl(-r\cc(\id_\g,r)+(r\cc(\id_\g,r))^{(1\,2)}\bigr)\\
	&\quad+\mu_3\cc(D,D,\id)\ot\bigl(l\cc(\omega,\id_M)-l\cc(\omega^{(1\,2)},\id_M)-l\cc(\id_\g,l)+(l\cc(\id_\g,l))^{(1\,2)}\bigr)\\
	&\quad+\mu_3\cc(D,\id,D)\ot\bigl(-r\cc(\omega,\id_M)+l\cc(\id_\g,r)-(r\cc(\id_\g,l))^{(1\,2)}+(r\cc(\id_\g,r))^{(1\,2)}\bigr)\\
	&\quad+\mu_3\cc(\id,D,D)\ot \bigl(r\cc(\omega^{(1\,2)},\id_M)+r\cc(\id_\g,l)-r\cc(\id_\g,r)-(l\cc(\id_\g,r))^{(1\,2)}\bigr).
\end{align*}
On the one hand, let $(M;l,r)$ be a representation of the Novikov algebra $(\g,\omega)$. Then applying \eqref{Nov-rep-1}-\eqref{Nov-rep-4} to the above equality, we obtain
\begin{align}
\rho\cc(\alpha(\nu),\id\ot\id_M)-\rho\cc(\id\ot\id_\g, \rho)+(\rho\cc(\id\ot\id_\g,\rho))^{(1\,2)}=0,\label{e1}
\end{align}
which implies that $(M;\rho)$ is a representation of the Lie algebra $(\g,\alpha(\nu))$ in the pseudo-tensor category $rNov\ot\mathcal{V}ec$.

On the other hand, let $(M;\rho)$ be a representation of the Lie algebra $(\g,\alpha(\nu))$ in the pseudo-tensor category $rNov\ot\mathcal{V}ec$. Then \eqref{e1} holds. Since the terms
$$
\mu_3\cc(D^2,\id,\id), \;\mu_3\cc(\id,\id,D^2),\;\mu_3\cc(D,D,\id),\;\mu_3\cc(D,\id,D)
$$
are linearly independent in $rNov(3)$, comparing their coefficients in the above expansion yields respectively \eqref{Nov-rep-3}, \eqref{Nov-rep-4}, \eqref{Nov-rep-1}, and \eqref{Nov-rep-2}. Hence $(M;l,r)$ is a representation of the Novikov algebra $(\g,\omega)$.
\end{proof}

Let $(\g,\omega)$ be a Novikov algebra and $(M;l,r)$ a representation of $(\g,\omega)$. By Lemma \ref{Nov-to-Lie}, we obtain a Lie algebra $(\g,\alpha(\nu))$ in $rNov\ot\mathcal{V}ec$ and a representation $(M;\rho)$ of $(\g,\alpha(\nu))$ in $rNov\ot\mathcal{V}ec$. By Definition \ref{CE-comp-wrt-mod}, $(\ce_{rNov\ot\mathcal{V}ec}^\bullet(\g;M),\hat{d}_\ce)$ is a cochain complex, which will be used to describe the cochain complex of the Novikov algebra $(\g,\omega)$ with coefficients in $(M;l,r)$.

Define the space of $n$-cochains of the Novikov algebra $(\g,\omega)$ with coefficients in $(M;l,r)$ by
\begin{align}\label{eq:Nov-n-cochain-coeff}
\frkC_{\nov}^n(\mathfrak{g};M)=\underset{(k_{n-1},\cdots,k_0)\in K_n}\bigoplus\Hom(\wedge^{k_{n-1}}\g\ot\cdots\ot\wedge^{k_0}\g,M),
\end{align}
where $K_n$ is defined by \eqref{K-n}.

\begin{pro}\label{Lie-Nov-1-1}
	For $n\geq 1$, $\ce_{rNov\ot \mathcal{V}ec}^n(\g;M)$ and $\frkC_{\nov}^n(\mathfrak{g};M)$ are isomorphic.
\end{pro}
\begin{proof}
	The argument is entirely parallel to that of Lemma \ref{Psi} and Proposition \ref{psi} in Section~2; we therefore only specify the isomorphism. 	
 For an arbitrary element
	\[
	F=\sum_{(p_1,\dots,p_n)\in P_n}\mu_n\circ(D^{p_1},\dots,D^{p_n})\otimes f^{(p_1,\dots,p_n)}\in \ce_{rNov\ot \mathcal{V}ec}^n(\g;M),
	\]
define a linear map $\Gamma_n:\ce_{rNov\ot \mathcal{V}ec}^n(\g;M)\to \frkC_{\nov}^n(\mathfrak{g};M)$ by
\begin{align}
	\Gamma_n(F)=\bigl\{f^{(p_1,\dots,p_n)}\bigr\}_{(p_1,\dots,p_n)\in P_n^{\ord}}.
\end{align}

Conversely, the inverse map ${\Gamma_n}^{-1}:\frkC_{\nov}^n(\mathfrak{g};M)\to \ce_{rNov\ot \mathcal{V}ec}^n(\g;M)$ is constructed as follows.  Given any
$$
f_{(k_{n-1},\dots,k_0)}\in\Hom(\wedge^{k_{n-1}}\g\ot\cdots\ot\wedge^{k_0}\g,M)\subset \frkC_{\nov}^n(\mathfrak{g};M),
$$
let $(p_1,\dots,p_n)=\phi_n^{-1}(k_{n-1},\dots,k_0)$, where $\phi_n^{-1}$ is given by Lemma \ref{P-K}.  Then
	\[
{\Gamma_n}^{-1}\bigl(f_{(k_{n-1},\dots,k_0)}\bigr)
	= \sum_{\sigma\in\Sh^{-1}(k_{n-1},\dots,k_0)} (-1)^\sigma \bigl(\mu_n\circ(D^{p_1},\dots,D^{p_n})\otimes f_{(k_{n-1},\dots,k_0)}\bigr)^\sigma .
	\]
	A direct verification shows that these two maps are inverses of each other, which completes the proof.
\end{proof}

\begin{thm}\label{cochain-Nov-coeff}
Let $(\g,\diamond)$ be a Novikov algebra and $(M;l,r)$ be its representation. Then $(\frkC_{\nov}^\bullet(\mathfrak{g};M),\hat{d}_{\nov})$ is a cochain complex,
where the space of $n$-cochains $\frkC_{\nov}^n(\g;M)$ is given by \eqref{eq:Nov-n-cochain-coeff}, and the coboundary operator $\hat{d}_\nov^n$ is given by
$$\hat{d}^n_\nov=\Gamma_{n+1}\circ \hat{d}^n_\ce\circ {\Gamma_n}^{-1},$$
which can be described as follows
$$
	\xymatrix@C=4em{
		\frkC_{\nov}^n(\mathfrak{g};M)
		\ar[r]^-{{\Gamma_{n}}^{-1}}
		\ar@{-->}[d]_{\hat{d}^n_{\nov}}
		&
		\ce^n_{rNov\ot\mathcal{V}ec}(\g;M)
		\ar[d]_{\hat{d}_\ce^n}
		\\
		\frkC_{\nov}^{n+1}(\mathfrak{g};M)
		&
		\ce^{n+1}_{rNov\ot\mathcal{V}ec}(\g;M)
		\ar[l]_-{\Gamma_{n+1}}
	}
$$
\end{thm}
\begin{proof}
Base on Lemma \ref{Nov-to-Lie}, the proof is entirely parallel to that of Theorem \ref{cochain-Nov} in Section \ref{sec:nov}, we shall not repeat it here.
\end{proof}

Using the same notations as in Section \ref{sec:nov}, we have the following explicit formulas for $\hat{d}_{\nov}^2$.
\begin{pro}\label{ex:coeff-2-d}
	For any $ f\in \Hom(\g\ot\g,M)$, we have $\hat{d}_{\nov}^2( f)= \{\hat{d}_{(1,0,2)}^{(1,1)}+d_{(0,2,1)}^{(1,1)}\}( f)$. For all $x_1,x_2,x_3\in \g$, $\hat{d}_{(1,0,2)}^{(1,1)}( f)$ and $\hat{d}_{(0,2,1)}^{(1,1)}( f)$ are given explicitly by
	\begin{eqnarray*}
		\big( \hat{d}_{(1,0,2)}^{(1,1)}( f)\big)(x_1;x_2,x_3)
		&=&r(x_2\ot f(x_1,x_3))-r(x_3\ot f(x_1,x_2))-f(x_1\diamond x_2,x_3)+f(x_1\diamond x_3,x_2),
	\end{eqnarray*}
	and
	\begin{eqnarray*}
		\big(\hat{d}_{(0,2,1)}^{(1,1)}( f)\big)(x_1,x_2;x_3)
		&=&l(x_1\ot f(x_2,x_3))-l(x_2\ot f(x_1,x_3))-r(x_3\ot f(x_1,x_2))+r(x_3\ot f(x_2,x_1))\\
		&&- f(x_1\diamond x_2,x_3)+ f(x_2\diamond x_1,x_3)- f(x_2,x_1\diamond x_3)+ f(x_1,x_2\diamond x_3).
	\end{eqnarray*}
\end{pro}
\begin{proof}
The proof is similar to that of Proposition \ref{pro:formular 2-d}, and we omit details.
\end{proof}
\begin{ex}\label{exmp:Polynomial}
Consider the series of Novikov algebras
$\g_\lambda $, $\lambda \in \mathds K$, constructed on the
 space $\mathds K[t,t^{-1}]$ of Laurent polynomials  with a basis $(e_k)_{k\in \mathds Z}$, $e_k=t^k$:
\[
e_n\diamond e_m = me_{n+m-1} + \lambda e_{n+m}.
\]
This is a particular case of the Gelfand-Dorfman construction from Example~\ref{ex-rN-N} relative to the ordinary derivation on $\mathds K[t,t^{-1}]$.
Let us show that $\huaH^2(\g_\lambda;\mathds K)=0$, where $\mathds K$
is considered as the trivial 1-dimensional representation of $\g_\lambda$.

First, assume $\lambda =0$, $\g=\g_0$.
Let $\kappa\in\Hom(\g\otimes\g,\K) $.
Then the conditions
$\hat d^{(1,1)}_{(0,2,1)}(\kappa)=0$
and
$\hat d^{(1,1)}_{(1,0,2)}(\kappa)=0$
turn into
\begin{gather}
    (m-n)\kappa(e_{n+m-1}, e_k) = k(\kappa (e_n, e_{m+k-1}) - \kappa(e_m, e_{n+k-1})), \label{eq:Z2-1Zero} \\
m\kappa(e_{n+m-1}, e_k) = k \kappa (e_{n+k-1}, e_m),
 \label{eq:Z2-2Zero}
\end{gather}
respectively.
It is easy to check that
for every sequence of scalars $(c_k)_{k\in \mathds Z}$,
the linear map $\kappa\in\Hom(\g\otimes\g,\K)$ defined by
$$ \kappa(e_n, e_m)= m c_{n+m-1}$$
is a 2-cocycle, i.e.
\eqref{eq:Z2-1Zero} and \eqref{eq:Z2-2Zero} hold.
Moreover, all 2-cocycles are of this form:
put $k=1$ into \eqref{eq:Z2-2Zero} to obtain
\[
\kappa(e_n,e_m) = m \kappa(e_{n+m-1}, e_1) = m c_{n+m-1}
\]
for all $n,m\in \mathds Z$.
Choose the 1-cochain $f(e_k) = -c_k$, $k\in \mathds Z$,
then
\[
\hat d^1_{\nov}(f)(e_n,e_m) = - f(e_n\diamond e_m) = -m f(e_{n+m-1}) = m c_{n+m-1} = \kappa(e_n,e_m).
\]

Next, assume $\lambda \ne 0$, $\g=\g_\lambda $.
The 2-cocycle conditions for $\kappa:\g\otimes \g\to \mathds K$
turn into
\begin{eqnarray}
    (m-n)\kappa(e_{n+m-1}, e_k) &=& k(\kappa (e_n, e_{m+k-1}) - \kappa(e_m, e_{n+k-1})) +\lambda (\kappa(e_n, e_{m+k})-\kappa(e_m,e_{n+k})), \label{eq:Z2-1nonZero} \\
m\kappa(e_{n+m-1}, e_k) &=& k \kappa (e_{n+k-1}, e_m) +\lambda(\kappa(e_{n+k},e_m)-\kappa(e_{n+m},e_k)).
 \label{eq:Z2-2nonZero}
\end{eqnarray}
Similarly, for every sequence of scalars $(c_k)_{k\in \mathds Z}$,
the linear map $\kappa\in\Hom(\g\otimes\g,\K)$ defined by
\[\kappa( e_n,e_m) = c_{n+m} + \frac{m}{\lambda}c_{n+m-1}
\]
satisfies both \eqref{eq:Z2-1nonZero} and \eqref{eq:Z2-2nonZero}.
As in the previous case, every 2-cocycle of $\g$ with coefficients in $\mathds K$ is of this form. Indeed,
put $k=0$ into \eqref{eq:Z2-2nonZero} to obtain
\[
\kappa(e_n,e_m) = \frac{m}{\lambda} \kappa(e_{n+m-1}, e_0) +\kappa(e_{n+m}, e_0),
\]
so $c_k = \kappa(e_k,e_0)$ works.
Choose the 1-cochain $f(e_k) = -\frac{1}{\lambda} c_k$, $k\in \mathds Z$,
then
\[
\hat d^1_{\nov}(f)(e_n,e_m) = - f(e_n\diamond e_m) = -m f(e_{n+m-1})-\lambda f(e_{n+m}) = \frac{m}{\lambda} c_{n+m-1} + c_{n+m}= \kappa(e_n,e_m).
\]
 $\huaH^2(\g_\lambda;\mathds K)=0$, for all $\lambda.$
\end{ex}
\begin{rmk}
The associated commutator algebra $\g_\lambda^{(-)}$ of the above Novikov algebra $\g_\lambda $
is isomorphic to the Witt Lie algebra, who has a non-trivial 2-cocycle
(the Gelfand-Fuchs cocycle) with scalar coefficients.
\end{rmk}

Using the same notations and following similar discussion as Theorem \ref{pro:pre-Lie-cochain}, we have the following result. The details are omitted.

\begin{thm}
  	Let $(\g,\diamond)$ be a Novikov algebra and $(M;l,r)$ its representation. Then the cochain complex $\frkC_{\pl}^\bullet(\g;M)$ of the pre-Lie algebra $(\g,\diamond)$ with the coefficients in $(M;l,r)$ is isomorphic to a quotient  of the cochain complex  $\frkC_{\nov}^\bullet(\g;M)$ of the Novikov algebra $(\g,\diamond)$ with the coefficients in $(M;l,r)$, i.e. we have the following short exact sequence of cochain complexes:
\begin{equation}
0 \to \overline{\frkC}_{\nov}^\bullet(\g;M) \stackrel{\iota}{\to}  {\frkC}_{\nov}^\bullet(\g;M) \stackrel{\pi}{\to} \frkC_{\pl}^\bullet(\g;M) \to 0,
\end{equation}
where $\iota$ and $\pi$ are the natural inclusion and projection. Consequently, there is a long exact sequence of the cohomology groups:
\begin{equation}
\cdots \to \overline{\huaH}^n_\nov(\g;M)   \stackrel{ \huaH^n(\iota)}{\to}  \huaH^n_\nov(\g;M) \stackrel{ \huaH^n(\pi)}{\to} \huaH^n_\pl(\g;M)   \stackrel{ c^n}{\to}   \overline{\huaH}^{n+1}_\nov(\g;M) \to\cdots,
\end{equation}
where the connecting map $c^n$ is given by
\begin{align*}
	\begin{split}
		c^n([f])= \left \{
		\begin{array}{ll}
			{[}0],                    & n=1,\\
			{[}\hat{d}_{(0,\cdots,0,1,n-2,2)}^{(0,\cdots,0,n-1,1)}(f)],     & n\geq 2,
		\end{array}
		\right.
	\end{split}
\end{align*}
for $[f]\in \huaH^n_\pl(\g;M)$.
\end{thm}

The following example shows that there is an essential distinction between the cohomology group of a Novikov algebra and the cohomology group of the underlying pre-Lie algebra.
\begin{ex}
Let $\g=(\g_0,\diamond)$ be the Novikov algebra given by Example \ref{exmp:Polynomial}, namely,
$
\g_0=\mathds K[t,t^{-1}]
$
with basis $\{e_n\mid n\in\mathds Z\}$ and multiplication
$$
e_n\diamond e_m=me_{n+m-1},
\quad n,m\in\mathds Z.
$$
We regard $\mathds K$ as the trivial representation of both the Novikov algebra $\g_0$ and its underlying pre-Lie algebra.  Let us show that $\huaH_{\pl}^2(\g_0;\mathds K)\neq0$.

Define a bilinear map
$
\kappa:\g_0\otimes\g_0\longrightarrow\mathds K
$
by
$$
\kappa(e_n,e_m)=mn\delta_{n+m,1},
\quad n,m\in\mathds Z.
$$
We claim that $\kappa$ is a non-trivial $2$-cocycle of the underlying pre-Lie algebra.
Indeed, the $2$-cocycle condition $\hat d_{\pl}^2(\kappa)=0$ is equivalent to \eqref{eq:Z2-1Zero}.
For all $m,n,k\in\mathds Z$, applying $\kappa$ to \eqref{eq:Z2-1Zero}, we have
\begin{align*}
&\quad(m-n)\kappa(e_{n+m-1},e_k)-k\bigl(\kappa(e_n,e_{m+k-1})-\kappa(e_m,e_{n+k-1})\bigr)\\
&=k(m-n)(n+m-1)\delta_{n+m+k,2}-k(m-n)(1-k)\delta_{n+m+k,2}=k(m-n)(n+m+k-2)\delta_{n+m+k,2}\\
&=0,
\end{align*}
which implies that $\kappa\in \ker(\hat d_{\pl}^2).$

We next show that $\kappa$ is not a coboundary. Suppose that there exists a $1$-cochain
$
f:\g_0 \longrightarrow \mathds K
$
such that $\hat d_{\pl}^1(f)=\kappa$. Since $\mathds K$ is the trivial representation,
$$
\hat d_{\pl}^1(f)(e_n,e_m)=-f(e_n\diamond e_m)=-m\,f(e_{n+m-1}).
$$
Taking $m=1$, we get
$
\kappa(e_n,e_1)=-f(e_n).
$
However,
$$
\kappa(e_n,e_1)=n\,\delta_{n+1,1}=n\,\delta_{n,0}=0
$$
for all $n\in\mathds Z$. Hence $f(e_n)=0$ for every $n$, so $f=0$. This would force $\kappa=0$, which is impossible since $\kappa(e_2,e_{-1})=-2\neq 0.$
Therefore,
$
0\neq[\kappa] \in \huaH_{\pl}^2(\g_0;\mathds K),
$
and consequently
$
\huaH_{\pl}^2(\g_0;\mathds K)\neq 0.
$

Example~\ref{exmp:Polynomial} gives $\huaH_{\nov}^2(\g_0;\mathds K)=0$. Hence
\[
	\huaH_{\pl}^2(\g_0;\mathds K)
	\not\cong
	\huaH_{\nov}^2(\g_0;\mathds K).
\]
This example illustrates that the second cohomology group of the underlying pre-Lie algebra need not agree with that of the Novikov algebra.

\end{ex}

\section{Abelian extensions of Novikov algebras}\label{sec:ext}

In this subsection, we study abelian extensions of a Novikov algebra and show that they are classified by the second cohomology group.

Let $(\g,\diamond)$ and $(\hat{\g},\hat{\diamond})$ be Novikov algebras. An exact sequence of Novikov algebra homomorphisms
\[
0 \longrightarrow M \overset{\iota}{\longrightarrow} \hat{\g} \overset{\pi}{\longrightarrow} \g \longrightarrow 0
\]
is called an \textbf{extension} of $\g$ by $M$. For such an extension, $M$ is regarded as the ideal $\ker(\pi)$ of $\hat{\g}$ via the embedding $\iota$.
\begin{defi}
\begin{itemize}
\item[{\rm(i)}] An extension is called \textbf{abelian} if $M$ is an abelian ideal of $\hat{\g}$, i.e. $M^2=0$.
\item[{\rm(ii)}] An extension is called \textbf{central} if $M\subset Z(\hat{\g})$,  i.e.
\begin{align*}
x\hat{\diamond} m=0=m\hat{\diamond} x,\quad\forall x\in \hat\g, m\in M.
\end{align*}
\item[{\rm(iii)}] A linear map $s : \g \to \hat{\g}$ is called a \textbf{section} of the extension if $\pi \circ s = \id_{\g}$.
\end{itemize}
\end{defi}

\begin{defi}
	Two extensions $\hat{\g}$ and $\hat{\g}'$ of $\g$ by $M$ are called \textbf{isomorphic} if there is an isomorphism of Novikov algebras such that the following commutative diagram commutes:
	$$
	\xymatrix{
		0 \ar[r] & M \ar[r]^{\iota} \ar@{=}[d]^{} & \hat{\g} \ar[r]^{\pi} \ar[d]^{\gamma}  & \g \ar[r] \ar@{=}[d]^{}& 0 \\
		0 \ar[r] & M \ar[r]^{\iota'} & \hat{\g}' \ar[r]^{\pi'} & \g \ar[r] & 0 &
	}
	$$
\end{defi}

Let $\hat{\g}$ be an abelian extension of $\g$ by $M$. Define two linear maps $l,r:\g\ot M\to M$ by
\begin{align}\label{abelian-exten-l-r}
l(x\ot m)=s(x)\hat{\diamond}m,\quad r(x\ot m)=m\hat{\diamond} s(x), \quad\forall x\in\g,m\in M.
\end{align}
\begin{lem}\label{lem-rep-extension}
With the above notations, $(M;l,r)$ is a representation of the Novikov algebra $(\g,\diamond)$ and
does not depend on the choice of sections $s$. Moreover, isomorphic abelian extensions give the same representation of $\g$ on $M$.
\end{lem}

\begin{proof}
 Since $\pi$ is a homomorphism of Novikov algebras, we have $s(x)\hat{\diamond}s(y)-s(x\diamond y)\in M$ for all $x,y\in\g$. Thus, $(s(x)\hat{\diamond}s(y)-s(x\diamond y))\hat{\diamond} m=0$. We have
	\begin{eqnarray*}
		&&l(x\diamond y\ot m)-l(y\diamond x\ot m)-l(x\ot l(y\ot m))+l(y\ot l(x\ot m))\\
		&=&s(x\diamond y) \hat{\diamond} m-s(y\diamond x)\hat{\diamond} m-s(x)\hat{\diamond}(s(y)\hat{\diamond} m)+s(y)\hat{\diamond}(s(x)\hat{\diamond} m)\\
		&=&(s(x)\hat{\diamond}s(y))\hat{\diamond}m-(s(y)\hat{\diamond}s(x))\hat{\diamond} m-s(x)\hat{\diamond}(s(y)\hat{\diamond} m)+s(y)\hat{\diamond}(s(x)\hat{\diamond} m)\\
		&\overset{\eqref{l-Nov-1}}{=}&0,
	\end{eqnarray*}
	which gives \eqref{Nov-rep-1}. The other identities \eqref{Nov-rep-2}-\eqref{Nov-rep-4} are obtained similarly. Therefore, $(M;l,r)$ is a representation of $(\g,\diamond)$.

	Next, let $s':\g\to\hat{\g}$ be another section. It follows from $\pi(s'(x)-s(x))=0$ that $s'(x)-s(x)\in M$ for all $x\in \g$. Then we have
	\begin{align}
		l'(x\ot m)&=s'(x)\hat{\diamond} m=(s'(x)-s(x)+s(x))\hat{\diamond} m=s(x)\hat{\diamond} m=l(x\ot m),\label{e-1}\\
		r'(x\ot m)&=m\hat{\diamond}s'(x)=m\hat{\diamond}(s'(x)-s(x)+s(x))=m\hat{\diamond}s(x)=r(x\ot m).\label{e-2}
	\end{align}
	Hence $l$ and $r$ do not depend on the choice of sections.
	
Finally, assume that $\hat{\g}$ and $\hat{\g}_1$ are isomorphic abelian extensions, and $\gamma:\hat{\g}\to \hat{\g}_1$ is the isomorphism of Novikov algebras satisfying such that $\gamma\circ\iota=\iota_1, \pi=\pi_1\circ\gamma$.
Choose sections $s$ and $s_1$ of $\pi$ and $\pi_1$, respectively. For any $x\in\g$, we have $\pi_1(s_1(x))=x=\pi_1(\gamma(s(x)))$, hence $s_1(x)-\gamma(s(x))\in M$. Then, for every $m\in M$,
\begin{align*}
	s_1(x)\hat{\diamond} m
	&= (s_1(x)-\gamma(s(x))+\gamma(s(x)))\hat{\diamond} m
	= \gamma(s(x))\hat{\diamond} m
	= \gamma(s(x)\hat{\diamond} m)
	= s(x)\hat{\diamond} m,\\
	m\hat{\diamond} s_1(x)
	&= m\hat{\diamond} (s_1(x)-\gamma(s(x))+\gamma(s(x)))
	= m\hat{\diamond} \gamma(s(x))
	= \gamma(m\hat{\diamond} s(x))
	= m\hat{\diamond} s(x).
\end{align*}
Therefore, isomorphic abelian extensions give rise to the same representation.
\end{proof}

Let $0\to M\to \hat{\g}\to \g\to0 $ be an abelian extension and $s:\g\to\hat{\g}$ a section. Define a linear map $\kappa:\g\ot\g\to M$ by
\begin{align}\label{abelian-exten-kap}
\kappa(x,y):=s(x)\hat{\diamond}s(y)-s(x\diamond y), \quad \forall x,y\in \g.
\end{align}
$\kappa$ is well-defined since $\pi(\kappa(x,y))=0$.

\begin{pro}
With the above notations, $\kappa$ is a $2$-cocycle on the Novikov algebra $(\g,\diamond)$ with coefficients in the representation $(M;l,r)$ given by Lemma \ref{lem-rep-extension}. Moreover, its cohomology class in $\huaH^2(\g;M)$ does not depend on the choice of the section $s$.
\end{pro}

\begin{proof}
We need to show that $\hat{d}_{\mathrm{Nov}}^2(\kappa)=0$. Equivalently, by Proposition \ref{ex:coeff-2-d}, both components vanish:
	\[
	\hat{d}_{(1,0,2)}^{(1,1)}(\kappa)=0,\qquad \hat{d}_{(0,2,1)}^{(1,1)}(\kappa)=0.
	\]
	For all $x_1,x_2,x_3\in\g$, we have
	\begin{eqnarray*}
		&&\bigl(\hat{d}_{(1,0,2)}^{(1,1)}(\kappa)\bigr)(x_1;x_2,x_3)\\
		&=&r(x_2\ot\kappa(x_1,x_3))-r(x_3\ot\kappa(x_1,x_2))-\kappa(x_1\diamond x_2,x_3)+\kappa(x_1\diamond x_3,x_2)\\
		&=&(s(x_1)\hat{\diamond} s(x_3))\hat{\diamond} s(x_2)-s(x_1\diamond x_3)\hat{\diamond}s(x_2)-(s(x_1)\hat{\diamond} s(x_2))\hat{\diamond} s(x_3)+s(x_1\diamond x_2)\hat{\diamond}s(x_3)\\
		&&-s(x_1\diamond x_2)\hat{\diamond} s(x_3)+s((x_1\diamond x_2)\diamond x_3)+s(x_1\diamond x_3)\hat{\diamond} s(x_2)-s((x_1\diamond x_3)\diamond x_2)\\
		&=&(s(x_1)\hat{\diamond} s(x_3))\hat{\diamond} s(x_2)-(s(x_1)\hat{\diamond} s(x_2))\hat{\diamond} s(x_3)+s((x_1\diamond x_2)\diamond x_3)-s((x_1\diamond x_3)\diamond x_2)\\
		&\overset{\eqref{l-Nov-2}}{=}&0,
	\end{eqnarray*}
	and
	\begin{eqnarray*}
		&&\big(\hat{d}_{(0,2,1)}^{(1,1)}( \kappa)\big)(x_1,x_2;x_3)\\
		&=&l(x_1\ot \kappa(x_2,x_3))-l(x_2\ot \kappa(x_1,x_3))-r(x_3\ot \kappa(x_1,x_2))+r(x_3\ot \kappa(x_2,x_1))\\
		&&- \kappa(x_1\diamond x_2,x_3)+ \kappa(x_2\diamond x_1,x_3)- \kappa(x_2,x_1\diamond x_3)+ \kappa(x_1,x_2\diamond x_3)\\
		&=& s(x_1)\hat{\diamond}\bigl(s(x_2)\hat{\diamond}s(x_3)\bigr)
		- s(x_1)\hat{\diamond}s(x_2\diamond x_3) - s(x_2)\hat{\diamond}\bigl(s(x_1)\hat{\diamond}s(x_3)\bigr)+ s(x_2)\hat{\diamond}s(x_1\diamond x_3) \\
		&&- \bigl(s(x_1)\hat{\diamond}s(x_2)\bigr)\hat{\diamond}s(x_3)
		+ s(x_1\diamond x_2)\hat{\diamond}s(x_3)+ \bigl(s(x_2)\hat{\diamond}s(x_1)\bigr)\hat{\diamond}s(x_3)- s(x_2\diamond x_1)\hat{\diamond}s(x_3) \\
		&&- s(x_1\diamond x_2)\hat{\diamond}s(x_3)
		+ s((x_1\diamond x_2)\diamond x_3) + s(x_2\diamond x_1)\hat{\diamond}s(x_3)- s((x_2\diamond x_1)\diamond x_3) \\
		&&- s(x_2)\hat{\diamond}s(x_1\diamond x_3)
		+ s(x_2\diamond(x_1\diamond x_3)) + s(x_1)\hat{\diamond}s(x_2\diamond x_3)- s(x_1\diamond(x_2\diamond x_3))\\
		&=& s(x_1)\hat{\diamond}\bigl(s(x_2)\hat{\diamond}s(x_3)\bigr)
		- s(x_2)\hat{\diamond}\bigl(s(x_1)\hat{\diamond}s(x_3)\bigr) - \bigl(s(x_1)\hat{\diamond}s(x_2)\bigr)\hat{\diamond}s(x_3)
		+ \bigl(s(x_2)\hat{\diamond}s(x_1)\bigr)\hat{\diamond}s(x_3)\\
		&& + s((x_1\diamond x_2)\diamond x_3)
		- s((x_2\diamond x_1)\diamond x_3) + s(x_2\diamond(x_1\diamond x_3))
		- s(x_1\diamond(x_2\diamond x_3))\\
		&\overset{\eqref{l-Nov-1}}{=}&0.
	\end{eqnarray*}
	Therefore $\kappa$ is a $2$-cocycle.
	
	Next, let $s':\g\to\hat{\g}$ be another section. Define a linear map $t:=s-s':\g\to M$. $t$ is well-defined. Indeed, $\pi\circ t=\pi\circ(s-s')(x)=0$ for all $x\in\g$, which implies that $t(\g)\subset M$.
		On the other hand, by $s=s'+t$, we have
	\begin{eqnarray*}
		\kappa(x_1,x_2)-\kappa'(x_1,x_2)&=&s(x_1)\hat{\diamond} s(x_2)-s(x_1\diamond x_2)-s'(x_1)\hat{\diamond} s'(x_2)+s'(x_1\diamond x_2)\\
		&=&(s'+t)(x_1)\hat{\diamond} (s'+t)(x_2)-(s'+t)(x_1\diamond x_2)-s'(x_1)\hat{\diamond} s'(x_2)+s'(x_1\diamond x_2)\\
		&=&s'(x_1)\hat{\diamond}t(x_2)+t(x_1)\hat{\diamond} s'(x_2)-t(x_1\diamond x_2)+t(x_1)\hat{\diamond}t(x_2)\\
		&=&l(x_1\ot t(x_2))+r(x_2\ot t(x_1))-t(x_1\diamond x_2),
	\end{eqnarray*}
	which implies that $\kappa-\kappa'=\hat{d}_\nov^1(t)$. Therefore, $\kappa$ and $\kappa'$ are in the same cohomology class.
\end{proof}

Let $\hat{\g}$ be an abelian extension of $\g$ by $M$. For a section $s:\g\to\hat{\g}$, the representation $(M;l,r)$ gives a Novikov algebra structure on $\g\oplus M$ as follows:
\begin{align*}
	(x+m)\bullet (y+u)=x\diamond y+l(x\ot u)+r(y\ot m)+\kappa(x,y),\quad\forall x,y\in\g,\ m,u\in M,
\end{align*}
where $l,r,\kappa$ are defined by \eqref{abelian-exten-l-r} and \eqref{abelian-exten-kap}.
In the sequel, we only consider abelian extensions of the form
\[
\g\oplus_{l,r,\kappa} M:=(\g\oplus M,\bullet).
\]

\begin{thm}
	Two abelian extensions $\g\oplus_{l,r,\kappa} M$ and $\g\oplus_{l,r,\kappa'} M$ are isomorphic if and only if $\kappa$ and $\kappa'$ are in the same cohomology class.
\end{thm}

\begin{proof}
Let $\g\oplus_{l,r,\kappa} M=(\g\oplus M,\bullet)$ and $\g\oplus_{l,r,\kappa'} M=(\g\oplus M,\bullet')$ be isomorphic abelian extensions, and let $\gamma:\g\oplus_{l,r,\kappa} M\to \g\oplus_{l,r,\kappa'} M$ be the corresponding isomorphism. Then there exists a linear map $f:\g\to M$ such that
\[
\gamma(x+m)=x+f(x)+m,\quad \forall x\in\g,\ m\in M.
\]
For all $x,y\in\g$, we have
\begin{align*}
		\gamma(x\bullet y)&=x\diamond y+f(x\diamond y)+\kappa(x,y),\\
		\gamma(x)\bullet'\gamma(y)&=x\diamond y+l(x\ot f(y))+r(y\ot f(x))+\kappa'(x,y).
\end{align*}
Since $\gamma$ is an isomorphism of Novikov algebras, we have $\gamma(x\bullet y)-\gamma(x)\bullet'\gamma(y)=0$. Thus, we obtain
\begin{align*}
		\kappa(x,y)-\kappa'(x,y)&=l(x\ot f(y))+r(y\ot f(x))-f(x\diamond y)=\hat{d}_{\mathrm{Nov}}^1(f),
\end{align*}
which means that $\kappa$ and $\kappa'$ are in the same cohomology class. The converse can be proved similarly.
\end{proof}

\begin{cor}
The Novikov algebra $\g_\lambda $ from Example~\ref{exmp:Polynomial}
has no non-trivial  {central} extensions.
\end{cor}



\noindent
{\bf Acknowledgements. } This research is supported by  NSF of Jilin Province (20260101013JJ) and NSFC (12471060, W2412041).  The first author was supported by RFS (project 25-41-00005).

\smallskip
\noindent
{\bf Declaration of interests. } The authors have no conflicts of interest to disclose.

\smallskip
\noindent
{\bf Data availability. } Data sharing is not applicable as no data were created or analyzed.

\end{document}